\documentclass[10pt]{article}
\usepackage{schola-article}

\title{Proof Theory and Interpolation for Sacchetti's Logics\thanks{Borja Sierra Miranda and Thomas Studer are supported by the Swiss National Science Foundation project 200021\_214820 \emph{Non-wellfounded and cyclic proof theory}.}}
\author{Borja Sierra Miranda\\
\vspace{-0.3em}
\small Logic and Theory Group\\
\vspace{-0.3em}
\small University of Bern\\
\small Bern, Switzerland\\
\small \texttt{borja.sierra@unibe.ch}
\and
Thomas Studer\\
\vspace{-0.3em}
\small Logic and Theory Group\\
\vspace{-0.3em}
\small University of Bern\\
\small Bern, Switzerland\\
\small \texttt{thomas.studer@unibe.ch}
}
\date{}

\usepackage{natbib}
\usepackage{modalops}
\usepackage{multicol}
\usepackage{quiver}

\newcommand\union{\cup}
\newcommand\Union{\bigcup}
\newcommand\inter{\cap}

\newcommand\set[1]{\left\{#1\right\}}

\DeclareMathOperator\domain{Dom}

\newcommand\function[1][]{\xlongrightarrow{#1}}

\DeclareMathOperator\length{lgth}

\newcommand\var{\mathrm{Var}}
\newcommand\wff{\mathrm{Fm}}

\newcommand\PA{\mathsf{PA}}
\newcommand\wGL[1]{\mathsf{wGL}_{#1}}
\newcommand\GL{\mathsf{GL}}
\newcommand\wKT[1]{\mathsf{wK4}_{#1}}
\newcommand\K{\mathsf{K}}
\newcommand\KT{\mathsf{K4}}

\newcommand\Grz{\mathsf{Grz}}
\newcommand\wGrz{\mathsf{wGrz}}

\newcommand\g[1]{\mathcal{G}#1}
\newcommand\n[1]{\mathcal{G}^\infty#1}

\newcommand\nmod[1]{\mathcal{G}^\infty_{\mathrm{mod}}#1}

\newcommand\Kax[1][]{\mathrm{K}_{#1}}
\newcommand\Lax[1][]{\mathrm{L}_{#1}}
\newcommand\wMax[1][]{\mathrm{4}_{#1}}

\newcommand\MP{\mathrm{MP}}
\newcommand\NEC[1][]{\mathrm{Nec}_{#1}}

\newcommand\ax{\mathrm{ax}}
\newcommand\Ax{\mathrm{Ax}}
\newcommand\botL{{\bot}\mathrm{L}}
\newcommand\botR{{\bot}\mathrm{R}}
\newcommand\toL{{\to}\mathrm{L}}
\newcommand\toR{{\to}\mathrm{R}}
\newcommand\negL{{\neg}\mathrm{L}}
\newcommand\negR{{\neg}\mathrm{R}}
\newcommand\veeL{{\vee}\mathrm{L}}
\newcommand\veeR{{\vee}\mathrm{R}}
\newcommand\wedgeL{{\wedge}\mathrm{L}}
\newcommand\wedgeR{{\wedge}\mathrm{R}}

\newcommand\modal[2][]{\nec^{#2}_{#1}}
\newcommand\modalalt[1]{\nec^{#1}_{\mathrm{mod}}}

\newcommand\lob{\mathrm{L}\ddot{\mathrm{o}}\mathrm{b}}

\newcommand\wk{\mathrm{Wk}}
\newcommand\ctr{\mathrm{Ctr}}
\newcommand\cut{\mathrm{Cut}}

\newcommand\localCut{\ell\mathrm{Cut}}
\newcommand\inv{\mathrm{Inv}}

\newcommand\rep{\mathrm{Rep}}
\newcommand\emp{\mathrm{Emp}}

\newcommand\voc{\mathrm{Voc}}

\newcommand\dep{\mathrm{dep}}
\newcommand\lhg{\mathrm{lhg}}
\newcommand\lrul{\mathrm{lRul}}
\newcommand\hg{\mathrm{hg}}
\newcommand\sub{\mathrm{Sub}}

\newcommand\mmod[1]{[#1]_{\mathrm{mod}}}

\begin{document}

\maketitle


\begin{abstract}
We study the proof theory of Sacchetti's modal logics  \(\wGL{n}\),
a family of logics generalizing Gödel--Löb provability logic $\GL$
 by replacing transitivity with
$n$-transitivity.
We make three main contributions.
First, we solve an open problem of Iwata by providing an effective cut elimination procedure for  \(\wGL{n}\).  
Second, building on this result, we introduce a new non-wellfounded sequent calculus for
\(\wGL{n}\)
with an improved subformula property.
Third, using this calculus together with interpolation templates, we prove that \(\wGL{n}\) has the uniform Lyndon interpolation property,  substantially strengthening previous interpolation results for these logics.
\end{abstract}

\section{Introduction}
Provability logic is the branch of modal logic that studies the notion of mathematical proof through the lens of modal operators, by interpreting the modality $\Box$ as \emph{provable in some mathematical theory} (most commonly an arithmetical theory such as Peano Arithmetic~$\PA$).
A landmark result in the area is Solovay's characterization \cite{solovay} of the provability logic of $\PA$ as the modal logic $\GL$, a celebrated system with a rich theory and numerous remarkable properties.

One of the most important such properties is the existence of fixed points with respect to modalized variables, known as the Fixed Point Theorem.
This result, due independently to de Jongh (unpublished) and Sambin \cite{Sambin1976-SAMAEF, Valentini}, has attracted considerable attention and has been studied in a variety of settings (see \cite{lindstrom, Visser1981-VISAPL, visser-fixpoint, japanese-fixpoint, Mardaev1993}).
In \cite{sacchetti}, Sacchetti initiated a systematic study of modal logics satisfying the Fixed Point Theorem and introduced a family of logics, denoted $\wGL{n}$ for $n \geq 1$, that generalize $\GL$ by replacing transitivity with $n$-transitivity (so that $\wGL{1} = \GL$).
Sacchetti proved that each $\wGL{n}$ satisfies the Fixed Point Theorem \cite{sacchetti}.
An explicit method for computing these fixed points, which will play a central role for our results, was subsequently given in \cite{kurahashi-okawa}.
Although these logics did not originate in the study of provability, they have since been shown to be complete with respect to an appropriate arithmetical semantics \cite{provability-kurahashi}.

The proof theory of $\wGL{n}$ was first investigated in \cite{iwata}, where a wellfounded sequent calculus is defined and cut admissibility is established via a non-constructive argument through the Kripke semantics of $\wGL{n}$.
That work also introduces a non-wellfounded sequent calculus for $\wGL{n}$ and uses it to prove Lyndon interpolation.

The present paper makes the following contributions:
\begin{enumerate}
  \item We solve an open problem from \cite{iwata} by giving an \emph{effective} cut elimination procedure for the well-founded sequent calculus of $\wGL{n}$.
    The key tool is a non-wellfounded (cyclic) proof system together with the general cut elimination procedure developed in \cite{coalgebraic}.  
\item This is the first description of an \emph{effective} cut elimination procedure based on the method introduced in~\cite{coalgebraic}.   On a technical note, we remark that since the calculi for $\wGL{n}$ do not have a nice subformula property,   obtaining this result required new ideas and techniques, in particular for the translation from non-wellfounded to wellfounded proofs.
  \item Based on the cut-elimination result,  we introduce a new non-wellfounded sequent calculus for $\wGL{n}$ with an  improved subformula property.
  \item Using this improved calculus together with interpolation templates~\cite{interpolation-guillermo}, we establish \emph{uniform} Lyndon interpolation for $\wGL{n}$, a substantial strengthening of the interpolation property given in \cite{iwata}.  Note that since the systems are rather weak,  they do not have \emph{nice} fixed points (with respect to polarities).  Through a careful analysis,  however, we can establish a certain shape of the fixed points that gives us the Lyndon property.  This can be seen as a generalization of the proof that $\GL$ has  uniform Lyndon interpolation.
\end{enumerate}

As mentioned above, our work is heavily based on non-wellfounded proofs.  Such deductive systems have been successfully developed for a wide range of logics with explicit fixed points, see, e.g., \cite{Brotherston, ill-founded-intui-linear-time, guillermo-CTL, thomasCK, KokkinisStuder+2016+171+192, saurin, das, masterModality} among many others.
Shamkanov~\cite{shamkanovGl} observed that non-wellfounded proofs are also closely related to the provability logic~\(\GL\), even though this logic lacks explicit fixed points. Subsequently, he and Savateev extended this framework to encompass the logics \(\Grz\)~\cite{shamkanovGrz} and \(\wGrz\)~\cite{shamkanovwGrz}. 
Based on their results,  non-wellfounded systems have been introduced for many provability and interpretability logics~\cite{justus,proofth-interpretability,bimodal-ulip,gls-ulip}.
One notable application of non-wellfounded proofs lies in establishing interpolation properties, see, e.g., \cite{interpolation-guillermo, pdl-interpolation, converse-pdl-interpolation, uip-interpretability, gls-ulip}. We likewise leverage our systems to derive interpolation results for Sacchetti's logics.


\paragraph{Organization of the paper.}
Section~\ref{sec:preliminaries} introduces the necessary background in modal logic and proof theory.
Section~\ref{sec:wellfounded-calculus} recalls the wellfounded sequent calculus for $\wGL{n}$ and its equivalence,  using cut, to $\wGL{n}$.
Section~\ref{sec:cut-elimination} introduces a non-wellfounded calculus for $\wGL{n}$ and uses it to provide an effective cut elimination proof for the wellfounded calculus; cut elimination is then exploited to obtain an alternative non-wellfounded calculus with an improved subformula property.
Finally, Section~\ref{sec:interpolation} establishes the uniform Lyndon interpolation property for $\wGL{n}$.

\section{Preliminaries}
\label{sec:preliminaries}

\subsection{Sacchetti's Logics}

We fix an infinite countable set \(\var\) whose elements are called \emph{propositional variables}.
We will work with the usual modal language,  which is defined by the following Backus--Naur form:
\[
  \phi ::= p \mid \bot \mid \phi \to \phi \mid \nec \phi,
\]
where \(p \in \var\).
The expressions of this language are called (modal) formulas.  $\wff$ denotes the collection of all formulas.
As usual, we define
\[
  \neg \phi := \phi \to \bot, 
  \qquad
  \phi \vee \psi := \neg \phi \to \psi,
  \qquad
  \phi \wedge \psi := \neg (\neg \phi \vee \neg \psi),
  \qquad
  \pos \phi := \neg \nec \neg \phi.
\]
For \(n \in \mathbb{N}\), we define \(\nec^0 \phi := \phi\) and \(\nec^{m+1} \phi := \nec \nec^m \phi\).
The complexity of \(\phi\), denoted \(|\phi|\) will be the number of \(\to\) and \(\nec\) connectives occurring in \(\phi\).
Given a formula \(\phi\), we define its set of \emph{subformulas}, denoted \(\sub(\phi)\), as usual:
\begin{align*}
  &\sub(p) = \set{p},
  &&\sub(\bot) = \set{\bot}, \\
  &\sub(\phi \to \psi) = \set{\phi \to \psi} \union \sub(\phi) \union \sub(\psi), 
  &&\sub(\nec \phi) = \set{\nec \phi} \union \sub(\phi).
\end{align*}

A \emph{(normal) modal logic} is a set of formulas \(L\) such that
\begin{enumerate}
  \item every classical propositional tautology (in the language of modal formulas) belongs to \(L\);
  \item \((\Kax)\) \ \(\nec(\phi \to \psi) \to \nec \phi \to \nec \psi\) belongs to \(L\) for any formulas \(\phi,\psi\);
  \item \(L\) is closed under the rules
    \[
      \AxiomC{\(\phi \to \psi\)}
      \AxiomC{\(\phi\)}
      \RightLabel{\(\MP\)}
      \BinaryInfC{\(\psi\)}
      \DisplayProof
      \qquad
      \AxiomC{\(\phi\)}
      \RightLabel{\(\NEC\)}
      \UnaryInfC{\(\nec \phi\)}
      \DisplayProof
    \]
\end{enumerate}
\(\K\) is the smallest normal logic.
For a normal modal logic \(L\) and a set of formulas \(\Gamma\) we will write \(L \oplus \Gamma\) for the smallest normal modal logic containing the set \(\Gamma\).
We are ready to define Sacchetti's logics.

\begin{definition}
  The logic \(\wGL{n}\), for \(n \geq 1\), is defined as \(\K \oplus \set{\nec(\nec^n \phi \to \phi) \to \nec \phi \mid \phi \in \wff}\).
\end{definition}

Note that \(\wGL{1}\) is just the usual G\"odel--L\"ob logic \(\GL\).
We fix an arbitrary \(n \geq 1\) for the rest of the paper.  Thus all our results will work for any of Sacchetti's logics (including \(\GL\)).

The logic \(\wGL{n}\) is \(n\)-transitive, which is expressed in the following lemma.

\begin{lemma}[\cite{sacchetti}]
  \(\wGL{n} \vdash \nec \phi \to \nec^{n+1} \phi\).
\end{lemma}

\subsection{Interpolation and modal equational systems}
\label{subsec:interpolation-and-modal-equational-systems}
In this subsection, we will recall some concepts about uniform Lyndon interpolation and we will introduce some concepts related to fixed points that we will need for the calculation of interpolants.
A \emph{vocabulary} is just a set of propositional variables.

\begin{definition}
  Let \(\phi\) be a formula, we define its positive and negative vocabulary, denoted \(\voc_+(\phi)\) and \(\voc_-(\phi)\) respectively, as
  \begin{align*}
    &\voc_+(p) = \set{p}, 
    &&\voc_-(p) = \varnothing, \\
    &\voc_+(\bot) = \varnothing, 
    &&\voc_-(\bot) = \varnothing, \\
    &\voc_+(\phi \to \psi) = \voc_-(\phi) \union \voc_+(\psi), 
    &&\voc_-(\phi \to \psi) = \voc_+(\phi) \union \voc_-(\psi), \\
    &\voc_+(\nec \phi) = \voc_+(\phi), 
    &&\voc_-(\nec \phi) = \voc_-(\phi).
    \qedhere
  \end{align*}
\end{definition}

We will write \(\overline{+}\) to mean \(-\) and \(\overline{-}\) to mean \(+\).
From the definitions and using an induction on the complexity of \(\phi\), it is straightforward to prove the following lemma.

\begin{lemma}
  \label{substitution-and-polarity}
  Let \(\bar{p} = p_0,\ldots,p_{m-1}\) and \(\bar{q} = q_0,\ldots,q_{k-1})\) be propositional variables and \(\phi(\bar{p},\bar{q})\), \(\psi_0\),\ldots,\(\psi_{m-1}\), \(\chi_0\),\ldots,\(\chi_{k-1}\) be formulas. 
  We have that if \(\voc_-(\phi) \cap \bar{p} = \varnothing\) and \(\voc_+(\phi) \cap \bar{q} = \varnothing\), then for \(b \in \set{+,-}\),
  \[
    \voc_b(\phi(\psi_0,\ldots,\psi_{m-1}, \chi_0,\ldots,\chi_{k-1})) \subseteq
    \voc_b(\phi) \setminus\bar{p}\bar{q}
    \union \Union_{i < m} \voc_b(\psi_i) \union \Union_{j < k} \voc_{\overline{b}}(\chi_j).
  \]
\end{lemma}

Uniform Lyndon interpolation arises from combining two different strengthening of the interpolation property: uniform interpolation and Lyndon interpolation \cite{ulip-Kurahashi}.
We recall its definition.

\begin{definition}
  Let \(L\) be a modal logic.
  For any formula \(\phi\) and vocabularies \(V_+\), \(V_-\), we will say that \(\iota\) is a \emph{(uniform Lyndon) interpolant} of \(\phi\) over \((V_+,V_-)\) in \(L\) if
  \begin{enumerate}
    \item \(\voc_+(\psi) \subseteq V_+\) and \(\voc_-(\psi) \subseteq V_-\),
    \item \(L \vdash \phi \to \iota\), and
    \item for any \(\psi\) such that \(\voc_+(\psi) \subseteq V_+\) and \(\voc_-(\psi) \subseteq V_-\), we have that \(L \vdash \phi \to \psi\) implies \(L \vdash \iota \to \psi\).
  \end{enumerate}
  \(L\) is said to have \emph{uniform Lyndon interpolation} if for any formula \(\phi\) and vocabularies \(V_+, V_-\) there is a uniform Lyndon interpolant of \(\phi\) over \((V_+,V_-)\) in \(L\).
\end{definition}

It is easy to show that any two (uniform Lyndon) interpolants of a formula \(\phi\), over the same vocabulary, will be logically equivalent.
For that reason it is common to speak about \emph{the} interpolant of a formula.
In Section~\ref{sec:interpolation}, we will show that \(\wGL{n}\) has uniform Lyndon interpolation.
For that purpose it will be necessary to introduce several notions related to fixed points.
We start with the notion of depth of a variable in a formula, which was introduced in \cite{kurahashi-okawa}.

\begin{definition}
  Let \(\phi\) be a formula and \(p\) be a variable.
  Define the set \(\dep(\phi,p) \subseteq \mathbb{N}\) as
  \begin{align*}
    &\dep(p,p) = \set{0},
    &&\dep(q,p) = \varnothing, \text{ for }p \neq q, \\
    &\dep(\bot,p) = \varnothing,
    &&\dep(\phi_0 \to \phi_1, p) = \dep(\phi_0, p) \union \dep(\phi_1, p), \\
    &\dep(\nec \phi_0,p) = \set{m + 1\mid \dep(\phi_0,p)}. 
  \end{align*}
  We also define the set \(\dep_n(\phi,p) = \set{[m]_n \mid m \in \dep(\phi,p)} \subseteq \mathbb{Z}/\mathbb{Z}n\).
\end{definition}

From this definition and using induction on \(\phi\) and \(\psi\), we can show the following.

\begin{lemma}[\cite{kurahashi-okawa}]
  \label{substitution-and-depth}
  Let \(\phi(p,q)\), \(\psi(p)\) and \(\psi\) be formulas with \(p \neq q\).
  Then
  \begin{enumerate}
    \item \( \dep(\phi(p, \chi),p) = \dep(\phi,p) \union \set{m + k \mid m \in \dep(\phi,q), k \in \dep(\chi,p)} \).
    \item \( \dep(\psi(\chi),p) = \set{m + k \mid m \in \dep(\phi,p), k \in \dep(\psi,p)} \).
  \end{enumerate}
\end{lemma}

The notion of fixed point for the logic \(\wGL{n}\) has been studied previously:  at the introduction of \(\wGL{n}\) without explicit calculations \cite{sacchetti} and later in \cite{kurahashi-okawa} where an explicit calculation is given.
We recall the notion of fixed point of a formula in the following definition.
In addition, we  introduce some properties that our fixed points have to fulfill in order to behave well for interpolation.

\begin{definition}
  Let \(\phi(p), \psi\) be formulas.
  We say that \(\psi\) is a \emph{fixed point} of \(\phi\) with respect to \(p\) in \(L\) if \(\voc(\psi) \subseteq \voc(\phi) \setminus \set{p}\) and \(L \vdash \psi \leftrightarrow \phi(\psi)\).
  Additionally, a fixed point is said to be
  \begin{enumerate}
    \item \emph{Lyndon} if \(\voc_+(\psi) \subseteq \voc_+(\phi) \setminus \set{p}\) and \(\voc_-(\psi) \subseteq \voc_-(\phi) \setminus \set{p}\).
    \item \emph{Depth preserving} if for any variable \(q\), \(\dep_n(\psi,q) \subseteq \dep_n(\phi,q)\).
      \qedhere
  \end{enumerate}
\end{definition}

The fixed points in \(\wGL{n}\) of a formula \(\nec\phi\) with respect to \(p\) where \(\dep_n(\nec\phi,p) \subseteq \set{[0]_n}\) have a particularly easy shape.

\begin{lemma}[\cite{kurahashi-okawa}]
  \label{fixpoints-wGL-deep-0}
  Let \(\phi(p)\) be a formula such that \(\dep_n(\nec\phi(p), p) \subseteq \set{0}\).
  Then \(\nec \phi(\top)\) is a depth preserving Lyndon fixed point of \(\phi\) with respect to \(p\) in \(\wGL{n}\).
\end{lemma}
\begin{proof}
  From \cite{kurahashi-okawa} we have \(\wGL{n} \vdash \nec \phi(\top) \leftrightarrow \nec \phi(\nec \phi(\top))\).
  By Lemma~\ref{substitution-and-polarity} we obtain that \(\voc_b(\nec \phi(\top)) \subseteq \voc_b(\phi) \setminus \set{p}\), and by Lemma~\ref{substitution-and-depth} we obtain that for any variable \(q\), \(\dep_n(\nec \phi(\top),q) \subseteq \dep_n(\nec \phi(p), q)\) as \(\dep_n(\top,q) = \varnothing\).
\end{proof}

Calculating fixed points will not be enough for our purposes, we will need to be able to solve multiple fixed points at the same time.
For that reason we recall the notion of equational system, which is commonly used in the \(\mu\)-calculus (e.g.\ see \cite{muCalcEq}).
In the area of provability logics the equivalent result is usually called generalized fixed points (e.g.\ see \cite{on-the-proof-of-solovay}, called \(n\)-ary fixed point theorem).\footnote{\label{bekic-footnote} In fact, the idea that solving individual fixed points suffices to solve simultaneous fixed points is proven in greater generality in Beki\'c Theorem, see \cite{Bekić1984}.}
We say that \(\phi\) is said to be \emph{modalised in \(p\)} if \(p\) always occurs in \(\phi\) under the scope of a modality \(\nec\) and a \emph{depth specification} on a set \(V \subseteq \var\) is a function \(d : V \function \mathbb{Z}/\mathbb{Z}n\).

\begin{definition}[Kurahashi-Lyndon Equational systems]
  Let \(V_+,V_-\) be vocabularies, \(d\) be a depth function on \(V_+ \union V_-\) and \(\bar{p} = (p_0,\ldots,p_{m-1})\) be a finite sequence of pairwise different variables not occurring in \(V_+ \union V_-\).
  A \emph{Kurahashi-Lyndon equational system over \((\bar{p},V_+,V_-,d)\)} is a set of triples \(\mathcal{E} = \set{(p_i,b_i,\phi_i) \mid i < m}\) such that for each \(i < m\)
  \begin{enumerate}
    \item \(b_i \in \set{+,-}\) and \(\phi_i\) is a formula;
    \item \(\voc_{b_i}(\phi_i) \subseteq V_+ \union B_+\) and \(\voc_{\overline{b_i}}(\phi_i) \subseteq V_- \union B_-\), where \(B_+ = \set{p_j \mid j < m, b_j = +}\) and \\ \(B_- = \set{p_j \mid j < m, b_j = -}\); and
    \item for each \(q \in V_+ \union V_-\), \(\dep_n(\phi_i) \subseteq d(q)\).
  \end{enumerate}
  The elements of \(\bar{p}\) are called \emph{unknowns} and those of \(\mathcal{E}\) are called \emph{equations}.

  A \emph{solution in \(L\)} to \(\mathcal{E}\) is a sequence  \((\psi_0,\ldots,\psi_{m-1})\) such that for each \(i < m\),  we have
  \begin{multicols}{2}
    \begin{enumerate}
      \item \(\voc_{b_i}(\psi_i) \subseteq V_+\) and \(\voc_{\overline{b_i}}(\psi_i) \subseteq V_-\),
      \item \(\dep_n(\psi_i,q) \subseteq d(q)\) for \(q \in V_+ \union V_-\), and
      \item \(L \vdash \psi_i \leftrightarrow \phi_i[\psi_0/p_0,\ldots,\psi_{m-1}/p_{m-1}]\).
    \end{enumerate}
  \end{multicols}

  \(\mathcal{E}\) is said to be
  \begin{multicols}{2}
    \begin{enumerate}
      \item \emph{Solvable in \(L\)} if it has a solution in \(L\).
      \item \emph{Simple} if \(\phi_i\) is a \(\nec\)-formula for \(i < m\).
      \item \emph{Modalized} if \(\phi_{i}\) is modalized in \(p_0,\ldots,p_i\).
      \item \emph{Positive} if \(b_i = {+}\) for  \(i < m\).
      \item \emph{Of depth \(0\)} if for any \(i, j < m\) we have that \(\dep_n(\phi_i,p_j) \subseteq \set{[0]_n}\).
    \end{enumerate}
  \end{multicols}
  Given an equational system \(\mathcal{E}\) over \((\bar{p}, V_+, V_-,d)\) with a solution \((\psi_0,\ldots,\psi_{m-1})\), we will also call the substitution \((\cdot)^*\), where \(p_i^* = \psi_i\) and \(q^* = q\) for \(q\) not in \(\bar{p}\), a solution of \(\mathcal{E}\).
\end{definition}

In general, the shape of fixed points in \(\wGL{n}\) may be complex~\cite{kurahashi-okawa}.
However, as we saw at Lemma~\ref{fixpoints-wGL-deep-0}, the shape gets greatly simplified when the depth of the variable is  \(0\) (modulo \(n\)).
With this simple shape, it is easy to show that the polarity of variables is preserved, a fundamental fact if we want to use modal equational system to show uniform Lyndon interpolation.
For this reason, we will focus only on solving modal equational systems of depth \(0\).
The methodology of the proofs below is similar to the ones  presented in \cite{gls-ulip}, with additional consideration of the depth of the variables.

\begin{lemma}\label{simple-equation-systems-solvable}
  Simple Kurahashi-Lyndon equational systems of depth \(0\) have a solution in \(\wGL{n}\).
\end{lemma}
\begin{proof}
  Let \(\mathcal{E} = \set{(p_i,b_i,\nec \phi_i) \mid i < m}\) be a simple Lyndon equational system over \((\bar{p}, V_+, V_-,d)\), remember that \(B_+ = \set{p_i \mid i < m, b_i = {+}}\), \(B_- = \set{p_i \mid i < m, b_i = {-}}\).
  We proceed by induction on \(m\), the number of unknowns.

  Take \(\nec\phi_0(p_0,\ldots,p_{m-1})\).  Since \(\dep_n(\nec\phi_0,p_0) \subseteq \set{[0]_n}\),  we know it has a depth preserving Lyndon fixed point \(\psi_0\) with respect to \(p_0\) using Lemma~\ref{fixpoints-wGL-deep-0}.
  Note that if \(b_0 = b_i\) then \(p_0 \not \in \voc_-(\phi_i)\) and if \(b_0 \neq b_i\) then \(p_0 \not \in \voc_+(\phi_i)\).
  Then, using Lemma~\ref{substitution-and-polarity}, we have that 
  \begin{align*}
    &\voc_{b_i}(\nec \phi_i[\psi_0/p_0]) \subseteq \voc_{b_i}(\phi_i) \setminus \set{p_0} \union \voc_{b_i}(\psi_0) = \voc_{b_i}(\phi_i) \setminus \set{p_0} \union \voc_{b_0}(\psi_0) \subseteq V_+ \union B_+ \setminus \set{p_0},\\
    &\hspace{14cm}\text{if }b_0 = b_i, \\
    &\voc_{b_i}(\nec \phi_i[\psi_0/p_0]) \subseteq \voc_{b_i}(\phi_i) \setminus \set{p_0} \union \voc_{\overline{b_i}}(\psi_0) = \voc_{b_i}(\phi_i) \setminus \set{p_0} \union \voc_{b_0}(\psi_0) \subseteq V_+ \union B_+ \setminus \set{p_0},\\
    &\hspace{14cm}\text{if }b_0 \neq b_i.
  \end{align*}
  We have a similar argument to show that \(\voc_{\overline{b_i}}(\nec \phi_i[\psi_0/p_0]) \subseteq V_- \union B_- \setminus \set{p_0}\).
  By Lemma~\ref{substitution-and-depth}, we find that for any \(p_j \neq p_0\)
  \[
    \dep_n(\nec \phi_i[\psi_0/p_0],p_j) \subseteq \dep_n(\nec \phi_i,p_j) \union \set{ [k_0 + k_1]_n \mid k_0 \in \dep(\nec \phi_i,p_0), k_1 \in \dep(\psi_0, p_j)} \subseteq \set{[0]_n},
  \]
  as \(\dep_n(\nec \phi_i,p_j) \subseteq \set{[0]_n}\) and \(\dep_n(\psi_0,p_j) \subseteq \dep_n(\nec \phi_0, p_j) \subseteq \set{[0]_n}\).
  Note that for \(p_0\), as there are no occurrences of \(p_0\) in \(\nec \phi_i[\psi_0/p_0]\), so \(\dep_n(\nec \phi_i[\psi_0/p_0]) \subseteq \varnothing\).
  Finally, let \(q \in V_+ \union V_-\), so \(q \neq p_0\), then we have that
  \begin{multline*}
    \dep_n(\nec \phi_i[\psi_0/p_0],q) \subseteq \dep_n(\nec \phi_i,q) \union \set{[k_0 + k_1]_n \mid k_0 \in \dep(\nec \phi_i,p_0), k_1 \in \dep(\psi_0, q)} \subseteq \\
    \dep_n(\nec \phi_i,q) \union \set{[k_1]_n \mid k_1 \in \dep(\psi_0, q)} \subseteq \dep_n(\nec \phi_i,q) \union \dep_n(\nec \phi_0,q) \subseteq d(q),
  \end{multline*}
  where we used that \(\dep_n(\nec \phi_i,p_0) \subseteq \set{[0]_n}\) and \(\dep_n(\psi_0,q) \subseteq \dep_n(\nec \phi_0,q)\).

  So we have that \(\mathcal{E}' = \set{(p_i, b_i, \nec \phi_i[\psi_0/p_0]) \mid 1 \leq i < m}\) is a simple Lyndon equational system of depth~\(0\) over \((p_1 \cdots p_{n-1}, V_+,V_-,d)\) and it is solvable in \(\wGL{n}\) by the induction hypothesis, let \((\chi_1,\ldots,\chi_{m-1})\) be a solution of it.
  Let us define \(\chi_0 = \psi_0[\chi_1/p_1,\ldots,\chi_{m-1}/p_{m-1}]\), then we claim that \((\chi_0,\ldots,\chi_{m-1})\) is a solution of \(\mathcal{E}\).

  \emph{Conditions on polarity}.
  Note that we already have that \(\voc_{b_i}(\chi_i) \subseteq V_+\) and \(\voc_{\overline{b_i}}(\chi_i) \subseteq V_-\) for \(1 \leq i < m\).
  We notice that, by cases on the value of \(b_i\) and \(b_0\), if \(b_i = b_0\) then \(p_i \not\in \voc_-(\psi_0)\) and if \(b_i \neq b_0\) then \(p_i \not \in \voc_+(\psi_0)\) for \(1 \leq i < m\).
  Using Lemma~\ref{substitution-and-polarity} we have that
  \begin{align*}
    \voc_{b_0}(\chi_0) &\subseteq \voc_{b_0}(\psi_0) \setminus \set{p_1,\ldots,p_{m-1}} \union \Union_{\footnotesize \begin{matrix} 1 \leq i < m \\ b_i = b_0 \end{matrix}} \voc_{b_0}(\chi_i) \union \Union_{\footnotesize \begin{matrix} 1 \leq i < m \\ b_i \neq b_0 \end{matrix}} \voc_{\overline{b_0}}(\chi_i) \\
                       &\subseteq \voc_{b_0}(\nec \phi_0) \setminus \set{p_0,\ldots,p_{m-1}} \union \Union_{\footnotesize \begin{matrix} 1 \leq i < m \\ b_i = b_0 \end{matrix}} \voc_{b_i}(\chi_i) \union \Union_{\footnotesize \begin{matrix} 1 \leq i < m \\ b_i \neq b_0 \end{matrix}} \voc_{b_i}(\chi_i) \subseteq V_+.
  \end{align*}
  where we used that \(\voc_{b_0}(\nec \phi_0) \subseteq V_+ \union B_+\).
  An analogous argument shows that \(\voc_{\overline{b_0}}(\chi_0) \subseteq V_-\), so the desired polarities conditions hold.

  \emph{Conditions on depth}.
  Let \(1 \leq i < m\), then for any \(q \in V_+ \union V_-\) we have that \(\dep_n(\chi_i,q) \subseteq d(q)\) since \((\chi_1,\ldots,\chi_{n-1})\) is a solution of \(\mathcal{E}'\).
  Also, for \(\chi_0 = \psi_0[\chi_1/p_1,\ldots,\chi_{m-1}/p_{m-1}]\) we have the following
  \begin{align*}
    \dep_n(\chi_0, q) &\subseteq \dep_n(\psi_0,q) \union \Union_{1 \leq i < m} \set{[k_0 + k_1]_n \mid k_0 \in \dep(\psi, p_i), k_1 \in \dep(\chi_i,q)} \\
    &\subseteq  \dep_n(\psi_0,q) \union \Union_{1 \leq i < m} \set{[k_0 + k_1]_n \mid k_0 \in \dep(\nec\phi_0, p_i), k_1 \in \dep(\chi_i,q)} \\
    &\subseteq  \dep_n(\nec\phi_0,q) \union \Union_{1 \leq i < m} \set{[k_1]_n \mid k_1 \in \dep(\chi_i,q)} \subseteq d(q).
  \end{align*}
  where we used that \(\dep_n(\psi_0,p_i) \subseteq \dep_n(\nec\phi_0,p_i) \subseteq \set{[0]_n}\), \(\dep_n(\nec \phi_0,q) \subseteq d(q)\) and \(\dep_n(\chi_i,q) \subseteq d(q)\).
 
  \emph{Desired equivalences}.
  We have \(\wGL{n} \vdash \chi_i \leftrightarrow (\nec \phi_i[\psi_0/p_0])[\chi_1/p_1, \ldots, \chi_{n}/p_n]\) for \(1 \leq i < m\).
  Using  \(p_0 \neq p_i\) for \(1 \leq i < m\), we obtain  
  \begin{align*}
    (\phi_i[\psi_0/p_0])[\chi_1/p_1, \ldots, \chi_{m-1}/p_{m-1}] 
    &= \phi_i[\psi_0[\chi_1/p_1, \ldots, \chi_{m-1}/p_{m-1}]/p_0, \chi_1/p_1, \ldots, \chi_{m-1}/p_{m-1}]\\
    &= \phi_i[\chi_0/p_0,\chi_1/p_1, \ldots, \chi_{m-1}/p_{m-1}],
  \end{align*}
  as desired.
  It only remains to show \(\wGL{n} \vdash \chi_0 \leftrightarrow \nec\phi_0[\chi_0/p_0, \ldots, \chi_{m-1}/p_{m-1}]\).
  We have,  as \(\psi_0\) is a fixed point,  that \(L \vdash \psi_0 \leftrightarrow \nec\phi_0[\psi_0/p_0]\).
  Since \(\wGL{n}\) is closed under substitutions, we obtain that 
  \[\wGL{n} \vdash \chi_0 \leftrightarrow (\nec\phi_0[\psi_0/p_0])[\chi_1/p_1,\ldots,\chi_{m-1}/p_{m-1}],\]
  and we can use the same equalities as before because of \(p_0 \neq p_i\) for \(1 \leq i < m\).
\end{proof}

A formula \(\phi\) is \emph{positive in \(p\)} if \(p \not \in \voc_-(\phi)\).

\begin{theorem}\label{positive-modalized-equation-systems-solvable}
  We have that
  \begin{enumerate}
    \item Every formula \(\phi(p)\) that is positive and modalized in \(p\) such that \(\dep_n(\phi,p) \subseteq \set{[0]_n}\) has a depth preserving Lyndon fixed point in \(\wGL{n}\).
    \item Positive modalized Lyndon equation systems of depth \(0\) are solvable in \(\wGL{n}\).
  \end{enumerate}
\end{theorem}
\begin{proof}
  Proof of 1.\footnote{This proof follows the one given in \cite{lindstrom} for \(\GL\), with the addition of the condition on the polarity of variables and depth.}
  Let \(\phi(p)\) be a formula that is modalized in \(p\), positive in \(p\) and such that \(\dep_n(\phi,p) \subseteq \set{[0]_n}\).
  Then there is a formula \(\phi'(q_0,\ldots,q_{m-1}, r_0,\ldots,r_{k-1})\) such that  
  \begin{enumerate}
    \item \(\phi(p) = \phi'(\nec\psi_0(p),\ldots,\nec\psi_{m-1}(p),\nec\chi_0(p),\ldots,\nec\chi_{k-1}(p))\) for some \(\psi_0,\ldots,\psi_{m-1}, \chi_0, \ldots, \chi_{k-1}\);
    \item the variables in \(\bar{q},\bar{r}\) are fresh and pairwise distinct;
    \item \(\phi'\) does not contain any \(\nec\) modality; and
    \item \(\voc_+(\phi') \subseteq \voc_+(\phi) \setminus \set{p} \union \bar{q}\), \(\voc_-(\phi') \subseteq \voc_-(\phi) \setminus \set{p} \union \bar{r}\) (so \(p\) does not occur in \(\phi'\)).
  \end{enumerate}
  In particular, we get \(\bar{q} \cap \voc_-(\phi') = \varnothing\) and \(\bar{r} \cap \voc_+(\phi') = \varnothing\).
  By Lemma~\ref{substitution-and-polarity} we have that for \(b \in \set{+,-}\)
  \[
    \voc_b(\phi) = \voc_{b}(\phi') \setminus \bar{q} \bar{r} \union \Union_{i < m} \voc_b(\psi_i) \union \Union_{j < k} \voc_{\overline{b}}(\chi_j)
    \quad \text{so}
    \quad \voc_b(\psi_i) \subseteq \voc_b(\phi) \text{ and } \voc_b(\chi_j) \subseteq \voc_{\overline{b}}(\phi).
  \]
  As \(p \not \in \voc_-(\phi)\) we immediately obtain  that \(p \not \in \voc_-(\psi_i)\) and \(p \not \in \voc_+(\chi_j)\).

  Since \(\phi'\) does not contain any modality, we have that \(\dep_n(\phi,q_i) \subseteq \set{[0]_n}\) and \(\dep_n(\phi,r_j) \subseteq \set{[0]_n}\).
    By Lemma~\ref{substitution-and-depth} we have that for \(q \in \voc(\phi) \setminus \set{p}\)
    \begin{align*}
      \dep_n(\phi, q) \subseteq \dep_n(\phi',q) &\union \Union_{i < m} \set{[k_0 + k_1]_n \mid k_0 \in \dep(\phi',q_i), k_1 \in\dep(\nec\psi_i,q)} \\
                                                &\union \Union_{j < k} \set{[k_0 + k_1]_n \mid k_0 \in \dep(\phi',r_j), k_1 \in\dep(\nec\chi_j,q)} \\
                                                &\subseteq \dep_n(\phi',q) \union \Union_{i < m} \set{ [k_1]_n \mid k_1 \in\dep(\nec\psi_i,q)}\union \Union_{j < k} \set{ [k_1]_n \mid  k_1 \in\dep(\nec\chi_j,q)}.
    \end{align*}
    Hence \(\dep_n(\nec \psi_i, q) \subseteq \dep_n(\phi,q)\), \(\dep_n(\nec \chi_j,q) \subseteq \dep_n(\phi,q)\) and \(\dep_n(\phi',q) \subseteq \dep_n(\phi,q)\).

  Consider the following equational system
  \( \set{(q_i, +,\nec\psi_i(\phi')) \mid i < m} \union \set{(r_i,-, \nec\chi(\phi')) \mid i < k}.
  \)
  Using Lemma~\ref{substitution-and-polarity} together with the previous observations, it is straightforward to check that it is a simple Lyndon \(\big(\bar{q}\bar{r}, V_+, V_-, d\big)\)-equational system where \(V_+ = \voc_+(\phi)\setminus \set{p}\), \(V_- = \voc_-(\phi)\setminus \set{p}\) and \(d(q) = \dep_n(\phi,q)\).
  Then, by Lemma~\ref{simple-equation-systems-solvable},  it has a solution \((\cdot)^*\) in \(\wGL{n}\).
  Let us define \(\eta = \phi'(q^*_0, \ldots, q^*_{m-1}, r^*_0, \ldots,r^*_{k-1})\).  
By Lemma~\ref{substitution-and-polarity} we know that \(\voc_+(\eta) \subseteq \voc_+(\phi) \setminus \set{p}\) and \(\voc_-(\eta) \subseteq \voc_-(\phi) \setminus \set{p}\); and by Lemma~\ref{substitution-and-depth} \(\dep_n(\eta,q) \subseteq \dep_n(\phi,q)\).

  Finally, since \((\cdot)^*\) is a solution of the equational system, we have that \(\wGL{n} \vdash q^*_i \leftrightarrow \nec \psi_i(\eta)\) for \(i < n\) and \(\wGL{n} \vdash r^*_i \leftrightarrow \nec \chi_i(\eta)\) for \(i < m\).
  So we obtain (using propositional reasoning) that
  \[\wGL{n} \vdash \eta \leftrightarrow \phi'(\nec \psi_0(\eta), \ldots,\nec \psi_{m-1}(\eta), \nec \chi_0(\eta), \ldots,\nec \chi_{k-1}(\eta)). \]
  In other words, \(\wGL{n} \vdash \eta \leftrightarrow \phi(\eta)\), so \(\eta\) is a Lyndon fixed point of \(\phi\) with respect to \(p\).

  Proof of 2. The proof is similar to the second point of Lemma~\ref{simple-equation-systems-solvable} using that if \(\phi_i\) is modalized in \(p_0,\ldots,p_i\) then \(\phi_i[\psi/p_0]\) is also modalized in \(p_0,\ldots,p_i\).
\end{proof}

\subsection{Local progress proof theory}\label{subsec:local-progress-proof-theory}

Local progress proof theory has been introduced in~\cite{coalgebraic}.  We will quickly recall the basic definitions and main results.

Let us start by fixing the notions concerning trees.
Given a set \(X\), we will write \(X^*\) (\(X^+\)) for the set of (non-empty) finite sequences with elements in \(X\), \(\epsilon\)  denotes the empty sequence.
As usual, we  write \(w \leq v\) to mean that \(w\) is an initial prefix of \(v\) and \((w,v]\) to mean the set \(\set{u \in X^* \mid w < u \leq v}\).
A (finitely branching) tree on \(A\) is a function \(T\) whose image is contained in \(A\) and whose domain is a prefix-closed non-empty subset of \(\mathbb{N}^*\) such that for every \(w \in \domain(T)\) there is an unique natural number \(k\) (called the \emph{arity of \(w\) in~\(T\)}) such that \(wi \in \domain(T)\) iff \(i < k\).
The \(wi \in \domain(T)\) are also called the \emph{immediate successors of \(w\)}.
Note that the domain of any tree always contains the word \(\epsilon\), and it is called the \emph{root of \(T\)}.

The elements of \(\domain(T)\) are also called the \emph{nodes of \(T\)}, the \(0\)-ary nodes are called \emph{leaves},  and the rest of the nodes are called \emph{interior nodes}.
Finally, an \emph{infinite branch in a tree \(T\)} is sequence of nodes \((w_i)_{i \in \mathbb{N}}\) such that \(w_0 = \epsilon\) and for each \(i\) there is a \(j\) such that \(w_{i+1} = w_i j\).

We fix a set \(\text{Seq}\) whose elements we  call \emph{sequents}. 
A sequent rule is a subset of \(\text{Seq}^+\).
Given a rule \(R\), we say  \((S_0,\ldots,S_{n-1},S)\) is an \emph{instance of \(R\) with premises \(S_0,\ldots,S_{n-1}\) and conclusion \(S\)} if \mbox{\((S_0,\ldots,S_{n-1},S) \in R\)}.
A rule \(R\) is said to be \emph{\(n\)-ary} if \(R \subseteq \text{Seq}^{n+1}\).
We now introduce the kind of sequent calculi we are going to use, called \emph{local progress  calculi}\/~\cite{coalgebraic}.

\begin{definition}
  A \emph{local progress sequent calculus} is a pair \(\mathcal{G} = (\mathcal{R}, (L_R)_{R \in \mathcal{R}})\) where \(\mathcal{R}\) is a set of sequent rules and each \(L_R\) is a function that takes an instance \(r = (S_0,\ldots,S_{n-1},S)\) of \(R\) and returns a subset of \(\set{0,\ldots,n-1}\).
  \(\mathcal{G}\) is said to be \emph{wellfounded} if each \(L_R\) is the constant function returning \(\varnothing\).
\end{definition}

We are ready to define the notion of proof in a local progress calculus.
From the definition of a proof,  we can infer that a wellfounded sequent calculus is just a sequent calculus in the usual (wellfounded) proof theory.

\begin{definition}
  Given a local progress sequent calculus \(\mathcal{G} = (\mathcal{R}, (L_R)_{R \in \mathcal{R}})\), a \emph{preproof in \(\mathcal{G}\)} is a tree with labels in \(\text{Seq} \times \mathcal{R}\) such that for each node \(w\) with immediate successors \(w0,\ldots,w(n-1)\) we have that \((S_0,\ldots,S_{n-1},S) \in R\) where \(S\) is the sequent at \(w\), \(R\) is the rule at \(w\) and \(S_i\) is the sequent at \(wi\).

  Let \(w\) be a node in \(\pi\) with immediate successors \(w0,\ldots,w(n-1)\).
  We say that \(wi\) is a \emph{progressing node} if \(i \in L_R(r)\) where \(R\) is the rule at \(w\) and \(r = (S_0,\ldots,S_{n-1},S)\) where \(S\) is the sequent at \(w\) and \(S_i\) the sequent at \(wi\).
  A proof is a preproof in which any infinite branch has infinitely many progressing nodes.
\end{definition}

We will write \(\mathcal{G} \vdash S\) to express that \(S\) is provable in \(\mathcal{G}\) and \(\pi \vdash_{\mathcal{G}} S\) to mean that \(\pi\) is a proof of \(S\) in \(\mathcal{G}\), omitting the subscript \(_\mathcal{G}\) when it is clear from context.
It is standard to write an instance \((S_0,\ldots,S_{n-1},S)\) of a rule \(R\) as
\[
  \AxiomC{\(S_0\)}
  \AxiomC{\(\cdots\)}
  \AxiomC{\(S_{n-1}\)}
  \RightLabel{\(R\)}
  \TrinaryInfC{\(S\)}
  \DisplayProof
\]
Given a proof \(\pi\) in \(\mathcal{G}\) we  define its main local fragment as the finite tree obtained from cutting the tree at the first progressing nodes from the root (in particular removing the progressing nodes).
Figure~\ref{fig:fragments} shows a picture of how a proof in a local progress calculus looks like.
The gray triangle at the bottom is the main local fragment, while the circular nodes are non-progressing nodes and the square nodes progressing nodes.
The \emph{local height of \(\pi\)}, denoted \(\lhg(\pi)\), is the height of its main local fragment.
Thes set of \emph{local rules of \(\pi\)}, denoted \(\lrul(\pi)\), is the set of rules occurring in the main local fragment.
We  say that \(\pi\) is \emph{locally \(R\)-free} if \(R \not \in\lrul(\pi)\).
Note that if \(\mathcal{G}\) is a wellfounded calculus,  the notion of local height agrees with the usual notion of height and the local rules is the set of rules occurring in the proof.
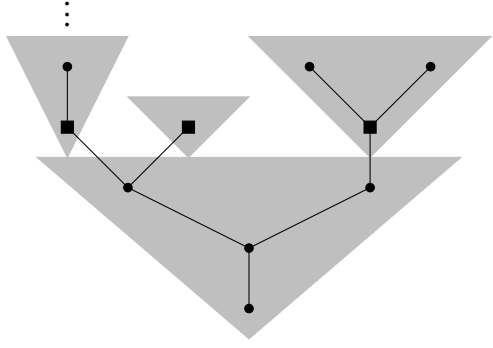
\begin{figure}
	\centering
	\begin{tikzpicture}[scale=0.8]
		\filldraw[gray!50] (0,-0.5) -- (3.5,2.5) -- (-3.5,2.5) -- cycle;
		\filldraw (0,0) circle (2pt);

		\filldraw (0,1) circle (2pt); \draw (0,0) -- (0,1);

		\filldraw[gray!50] (-3,2.5) -- (-4,4.5) -- (-2,4.5) -- cycle;
		\filldraw[gray!50] (-1,2.5) -- (-2,3.5) -- (0,3.5) -- cycle;
		\filldraw[gray!50] (2,2.5) -- (0,4.5) -- (4,4.5) -- cycle;
    \filldraw (-2,2) circle (2pt); \draw (0,1) -- (-2,2);
    \filldraw (2,2) circle (2pt); \draw (0,1) -- (2,2);

		\filldraw (-3.1,2.9) rectangle (-2.9,3.1); \draw (-2,2) -- (-3,3);
		\filldraw (-1.1,2.9) rectangle (-0.9,3.1); \draw (-2,2) -- (-1,3);
		\filldraw (1.9,2.9) rectangle (2.1,3.1); \draw (2,2) -- (2,3);

		\filldraw (-3,4) circle (2pt); \draw (-3,3) -- (-3,4); 
		\node at (-3,5) {\large\(\vdots\)};
		\filldraw (1,4) circle (2pt); \draw (2,3) -- (1,4); 
		\filldraw (3,4) circle (2pt); \draw (2,3) -- (3,4); 
	\end{tikzpicture}
	\caption{Structure of proofs in local progress calculi}
  	\label{fig:fragments}
\end{figure}

Given a local progress calculus \(\mathcal{G} = (\mathcal{R},(L_R)_{R \in \mathcal{R}})\) and a rule \(R'\), we define \(\mathcal{G} + R'\) as the local progress calculus with rules \(\mathcal{R} \union \set{R'}\) and \(L_{R'}\) the constant function returning \(\varnothing\).
Rules can interact with local progress calculi in different ways, the following definition introduces some of these ways and the theorem below shows that some of them are equivalent.

\begin{definition}\label{def:prop-of-rules}
  Let \(\mathcal{G}\) be a sequent calculus and \(R\) be a rule.
  We say that \(R\) is
  \begin{enumerate}
    \item \(R\) admissible if for any instance \((S_0,\ldots,S_{n-1},S) \in R\) we have that \(\pi_i \vdash_{\mathcal{G}} S_i\) for \(i < n\) implies the existence of \(\pi \vdash_{\mathcal{G}} S\).
      In addition we say that \(R\) is \emph{admissible preserving local height} if \(\lhg(\pi) \leq \max_{i < n}\lhg(\pi_i)\) and
        \emph{admissible preserving local rules} if \(\lrul(\pi) \subseteq \Union_{i < n}\lrul(\pi)\).
    \item \(R\) eliminable if \(\mathcal{G} + R\vdash S\) implies \(\mathcal{G} \vdash S\). 
    \item \(R\) locally admissible if for any instance \((S_0,\ldots,S_{n-1},S) \in R\) we have that if there are locally \(R\)-free \(\pi_i \vdash_{\mathcal{G}} S_i\)  for \(i < n\) then there is a locally \(R\)-free \(\pi \vdash_{\mathcal{G}} S\).
    \item An \(n\)-ary rule \(R\) is \(i\)-invertible  in \(\mathcal{G}\) (where \(i < n\)) if the rule
      \(
        \set{(S,S_i) \mid (S_0,\ldots,S_{n-1},S) \in R}
      \)
      is admissible.
      \(R\) is \emph{invertible} if it is \(i\)-invertible for each \(i < n\).
  \end{enumerate}
We will talk about invertibility preserving local height and/or local rules with the obvious meaning.
\end{definition}
A discussion of this terminology can be found in~\cite{bimodal-ulip}.  Note that according to our definition, eliminability of a rule means that the rule can be dropped from a calculus without affecting the provability of sequents.  If we want to talk about the effectiveness of a cut elimination procedure,  we will explicitly say effective cut elimination.

Given an admissible rule \(R\) in \(\mathcal{G}\),  an instance \((S_0,\ldots,S_{n-1},S) \in R\) and proofs \(\pi_i \vdash_{\mathcal{G}} S_i\) for \(i < n\), we  write \(R(\pi_0,\ldots,\pi_{n-1})\) to mean a proof of \(S\) in \(\mathcal{G}\) that exists by admissibility.
For the invertibility of a rule~\(R\) we  write  \(\inv{R}(\pi)\) instead.

The following is the key result of local progress proof theory (for details see \cite{proofth-interpretability,uip-interpretability}).

\begin{theorem}\label{local-adm-and-eliminability}
  Let \(\mathcal{G}\) be a local progress sequent calculus, then \(R\) is eliminable in \(\mathcal{G}\) iff \(R\) is locally admissible in \(\mathcal{G}\).
  If \(\mathcal{G}\) is wellfounded, then both are equivalent to \(R\) is admissible in \(\mathcal{G}\).
\end{theorem}

Finally, sometimes we  need to work with non-wellfounded proofs that have a particular finite representation.
To define this notion,  we introduce the notion of trees with backedges.
A \emph{tree with backedges} is an ordered pair \(\tau = (T,(\cdot)^\circ)\) such that \(T\) is a finite tree and \((\cdot)^\circ\) is a function from a subset of the leaves of~\(\tau\) to the nodes of \(\tau\) such that if \(w\) is in the domain then \(w^\circ < w\), the nodes in the domain of \((\cdot)^\circ\) are called \emph{repeat nodes}.

\begin{definition}
  Given a local progress calculus \(\mathcal{G} = (\mathcal{R}, (L_R)_{R \in \mathcal{R}})\), a \emph{cyclic preproof in \(\mathcal{G}\)} is a tree with backedges \(\pi = (T, (\cdot)^\circ)\) on \(\text{Seq} \times \mathcal{R}\) such that for any non-repeat node \(w\) with immediate successors \(w0\), \ldots, \(w(n-1)\), we have that \((S_0,\ldots,S_{n-1},S) \in R\), where \(S\) is the sequent at \(w\),  \(R\) is the rule at \(w\) and \(S_i\) is the sequent at \(wi\),
    and for any repeat node \(w\), \(w\) and \(w^\circ\) have the same sequent and rule.
    A \emph{cyclic proof in \(\mathcal{G}\)} is a cyclic preproof such that for any repeat node \(w\) there is a progressing node in \((w^\circ,w]\).
\end{definition}
In cyclic preproofs,  we  annotate repeat nodes with the word \(\rep\) at the right, instead of the rule at the node.

The \emph{main local fragment} of a cyclic proof is defined like for non-wellfounded proofs.
Note that it is always a tree without backedges.
A rule \(R\) will be \emph{cyclic admissible} if for any instance \((S_0,\ldots,S_{n-1},S) \in R\) we have that for any cyclic proofs \(\pi_i \vdash_{\mathcal{G}} S_i\) for \(i < n\), then there is a cyclic proof \(\pi \vdash_{\mathcal{G}} S\).
We define the notions of \emph{cyclic invertibility} and \emph{cyclic local admissibility} analogously.
For these notion we will also talk about preserving local height and local rules with its obvious meaning.

To work with cyclic proofs more easily, it is useful to introduce the following concept.
A cyclic proof~\(\pi\) is said to be \emph{locally backedge free} if for any \(w \in \rep(T)\) then \(w^\circ\) is not in the main local fragment of \(\pi\).
It is straightforward to see that any provable sequent with a cyclic proof has a locally backedge free cyclic proof, the process of obtaining it is effective and we will call it \emph{unfolding}.
Notice that unfolding preserves the local height and the local rules of a proof.

\section{Wellfounded calculus}
\label{sec:wellfounded-calculus}

A \emph{sequent} is an ordered pair \((\Gamma, \Delta)\) of finite multisets of formulas, usually denoted as \(\Gamma \Rightarrow \Delta\).
\(\Gamma\) is called the \emph{left side} of the sequent while \(\Delta\) is called the \emph{right side}.
Given a multiset of formulas \(\Gamma\) we will write \(\Gamma^s\) to mean the set obtained by deleting repetitions from \(\Gamma\).

The rules needed for the wellfounded calculus of \(\wGL{n}\) can be found at Figure~\ref{fig:rules-gwGL}.
The notation $\Box^n \Sigma$ denotes the multiset that results from prefixing each element of $\Sigma$ with $\Box^n$.
Further we use $\necd^n \Sigma$ for $\nec^n \Sigma, \Sigma$.
In the rules \((\botR)\), \((\toL)\), \((\toR)\) and \((\modal{\wGL{n}})\) the formula displayed at the conclusion is called the \emph{principal formula} (of the rule instance) and the displayed \(\nec\)-formula in the left side of the premise of the rule \((\modal{\wGL{n}})\) is called \emph{diagonal formula}.
In \((\modal{\wGL{n}})\) the formulas belonging to \(\nec \Sigma\) are called \emph{auxiliary formulas}.
Given a rule instance of \((\botR)\), \((\toL)\), \((\toR)\) or \((\modal{\wGL{n}})\) the auxiliary formulas together with the principal formula will be called the \emph{active formulas} (of the rule instance).
In \((\ax)\), \((\botL)\) and \((\modal{\wGL{n}})\) the multisets \(\Gamma\) and \(\Delta\) are called the \emph{weakening part} (of the rule instance).
Note that the weakening part of a rule instance can be changed arbitrarily and it will still be a rule instance of the same rule.

We have omitted the usual propositional rules \((\negL)\), \((\negR)\), \((\wedgeL)\), \((\wedgeR)\), \((\veeL)\) and \((\veeR)\).
They can easily be simulated with the rules \((\toL)\), \((\toR)\), \((\botL)\), \((\botR)\).
From now on we will use the rules \((\negL)\), \((\negR)\), \((\wedgeL)\), \((\wedgeR)\), \((\veeL)\), \((\veeR)\) as abbreviations for corresponding application of rules.

\begin{figure}
  \[
    \AxiomC{}
    \RightLabel{\(\ax\)}
    \UnaryInfC{\(p, \Gamma \Rightarrow p, \Delta\)}
    \DisplayProof
    \qquad
    \AxiomC{}
    \RightLabel{\(\botL\)}
    \UnaryInfC{\(\bot, \Gamma \Rightarrow \Delta\)}
    \DisplayProof
    \qquad
    \AxiomC{\(\Gamma \Rightarrow \Delta\)}
    \RightLabel{\(\botR\)}
    \UnaryInfC{\(\Gamma \Rightarrow  \bot,\Delta\)}
    \DisplayProof
  \]

  \[
    \AxiomC{\(\Gamma \Rightarrow \phi, \Delta\)}
    \AxiomC{\(\psi, \Gamma \Rightarrow \Delta\)}
    \RightLabel{\(\toL\)}
    \BinaryInfC{\(\phi \to \psi, \Gamma \Rightarrow \Delta\)}
    \DisplayProof
    \qquad
    \AxiomC{\(\phi, \Gamma \Rightarrow \psi, \Delta\)}
    \RightLabel{\(\toR\)}
    \UnaryInfC{\(\Gamma \Rightarrow \phi \to \psi, \Delta\)}
    \DisplayProof
  \]

  \[
    \AxiomC{\(\necd^n \Sigma, \nec^n \phi \Rightarrow \phi\)}
    \RightLabel{\(\modal{\wGL{n}}\)}
    \UnaryInfC{\(\nec \Sigma, \Gamma \Rightarrow \nec \phi, \Delta\)}
    \DisplayProof
  \]
  \caption{Sequent rules for \(\g{\wGL{n}}\)}
  \label{fig:rules-gwGL}
\end{figure}

\begin{definition}
  The wellfounded sequent calculus \(\g{\wGL{n}}\), for \(n \geq 1\), is given by the rules in Figure~\ref{fig:rules-gwGL}.\footnote{This calculus is a minor variation of the calculus introduced in \cite{iwata}.}
\end{definition}

By induction on the complexity of \(\phi\), we obtain the following lemma.
\begin{lemma}
  For any formula \(\phi\), \(\g{\wGL{n}} \vdash \phi, \Gamma \Rightarrow \phi, \Delta\).
\end{lemma}
When we use this lemma inside a proof in \(\g{\wGL{n}}\), we will annotate the sequent with the rule \((\Ax)\).

In Figure~\ref{fig:structural-rules}, we give the usual rules of weakening, contraction and cut.
We will show that our calculus behaves  structurally well, i.e.,  these rules are eliminable and, in the case of weakening and contraction, preserve height.
In addition, the propositional rules are invertible.
These basic facts are shown in the following lemma, while cut elimination is left for Section~\ref{sec:cut-elimination}, as the proof will be more involved.

\begin{figure}
  \[
    \AxiomC{\(\Gamma \Rightarrow \Delta\)}
    \RightLabel{\(\wk\)}
    \UnaryInfC{\(\Gamma, \Gamma' \Rightarrow \Delta, \Delta'\)}
    \DisplayProof
    \qquad
    \AxiomC{\(\Gamma, \Gamma' \Rightarrow \Delta,\Delta'\)}
    \RightLabel{\(\ctr\)}
    \UnaryInfC{\(\Gamma, (\Gamma')^s \Rightarrow \Delta, (\Delta')^s\)}
    \DisplayProof
  \]
  We recall that for a multiset \(\Theta\), \(\Theta^s\) is the set obtained by deleting repetitions in \(\Theta\).

  \[
    \AxiomC{\(\Gamma \Rightarrow \Delta, \chi\)}
    \AxiomC{\(\chi, \Gamma \Rightarrow \Delta\)}
    \RightLabel{\(\cut\)}
    \BinaryInfC{\(\Gamma \Rightarrow \Delta\)}
    \DisplayProof
  \]
  
  \caption{Structural rules}
  \label{fig:structural-rules}
\end{figure}

\begin{lemma}
  In \(\g{\wGL{n}}\) and \(\g{\wGL{n}} + \cut\) we have the following.
  \begin{enumerate}
    \item \(\wk\) is eliminable and admissible preserving height and rules.
    \item \(\botR\), \(\toL\) and \(\toR\) are invertible preserving height and rules.
    \item \(\ctr\) is eliminable and admissible preserving height and rules.
  \end{enumerate}
\end{lemma}
\begin{proof}
  The proof is by induction on the height of the proof.
\end{proof}

Thanks to the good structural behaviour of the calculus, it is easy to show the correspondence between \(\wGL{n}\) and \(\g{\wGL{n}}\).
First, we  show that any formula provable in \(\wGL{n}\) is also provable in \(\g{\wGL{n}}+ \cut\).
This was also established in \cite{iwata} simultaneously with  (non-effectively) cut admissibility.
Here we give a purely syntactic proof.

\begin{lemma}\label{from-wGL-to-gwGL}
  If \(\wGL{n} \vdash \phi\) then \(\g{\wGL{n}} + \cut \vdash { \Rightarrow \phi}\).
\end{lemma}
\begin{proof}
  It suffices to show that the set \(\set{\phi \mid \g{\wGL{n}} + \cut \vdash { \Rightarrow \phi}}\) contains the propositional tautologies, \((\Kax)\), \((\Lax[n])\) and is closed under \((\MP)\) and \((\NEC)\).
  That it contains all propositional tautologies is easy to show using the presence of the rules \((\ax)\), \((\botL)\), \((\toL)\) and \((\toR)\).
  Let us show the proofs of the other axioms:
  \[
    \AxiomC{\(\)}
    \RightLabel{\(\Ax\)}
    \UnaryInfC{\( \nec^n (\phi \to \psi), \necd^n \phi, \nec^n \psi \Rightarrow  \phi, \psi\)}
    \AxiomC{\(\)}
    \RightLabel{\(\Ax\)}
    \UnaryInfC{\( \psi, \nec^n (\phi \to \psi), \necd^n \phi, \nec^n \psi \Rightarrow  \psi\)}
    \RightLabel{\(\toL\)}
    \BinaryInfC{\( \necd^n (\phi \to \psi), \necd^n \phi, \nec^n \psi \Rightarrow  \psi\)}
    \RightLabel{\(\modal{\wGL{n}}\)}
    \UnaryInfC{\( \nec(\phi \to \psi), \nec \phi \Rightarrow  \nec \psi\)}
    \doubleLine\RightLabel{\(\toR\)}
    \UnaryInfC{\( \Rightarrow \nec(\phi \to \psi) \to \nec \phi \to \nec \psi\)}
    \DisplayProof
  \]
  \vspace{0.3cm}
  \[
    \AxiomC{}
    \RightLabel{\(\Ax\)}
    \UnaryInfC{\( \nec^n(\nec^n \phi \to \phi), \nec^n \phi \Rightarrow \nec^n \phi, \phi\)}
    \AxiomC{}
    \RightLabel{\(\Ax\)}
    \UnaryInfC{\( \phi, \nec^n(\nec^n \phi \to \phi), \nec^n \phi \Rightarrow \phi\)}
    \RightLabel{\(\toL\)}
    \BinaryInfC{\( \necd^n(\nec^n \phi \to \phi), \nec^n \phi \Rightarrow \phi\)}
    \RightLabel{\(\modal{\wGL{n}}\)}
    \UnaryInfC{\( \nec(\nec^n \phi \to \phi) \Rightarrow  \nec \phi\)}
    \RightLabel{\(\toR\)}
    \UnaryInfC{\( \Rightarrow \nec(\nec^n \phi \to \phi) \to \nec \phi\)}
    \DisplayProof
  \]
  
  For closure under \((\MP)\) assume that \(\pi \vdash { \Rightarrow \phi \to \psi}\) and \(\tau \vdash { \Rightarrow \phi}\) are proofs in \(\g{\wGL{n}} + \cut\).
  Then the desired proof is
  \[
    \AxiomC{\(\wk(\tau)\)}
    \noLine
    \UnaryInfC{\( \Rightarrow \phi, \psi\)}
    \AxiomC{\(\inv_{\toR}(\pi)\)}
    \noLine
    \UnaryInfC{\( \phi \Rightarrow \psi\)}
    \RightLabel{\(\cut\)}
    \BinaryInfC{\( \Rightarrow \psi\)}
    \DisplayProof
  \]
  For closure under necessitation assume that \(\pi \vdash { \Rightarrow \phi}\) in \(\g{\wGL{n}} + \cut\).
  Then the desired proof is
  \[
    \AxiomC{\(\wk(\pi)\)}
    \noLine
    \UnaryInfC{\(\nec^n \phi \Rightarrow \phi\)}
    \RightLabel{\(\modal{\wGL{n}}\)}
    \UnaryInfC{\(\Rightarrow \nec \phi\)}
    \DisplayProof
  \]
\end{proof}

For the converse direction of the theorem we will need to translate sequents into formulas.
The translation is defined as follows.
\[
  (\Gamma \Rightarrow \Delta)^\flat := \bigwedge \Gamma \to \bigvee \Delta.
\]
We will say that a rule \(R\) is \emph{sound in \(\wGL{n}\)} if for any instance \((S_0,\ldots,S_{n-1},S) \in R\) we have that \(\wGL{n} \vdash S^\flat_{i}\) for \(i < n\) implies \(\wGL{n} \vdash S^\flat\).

\begin{lemma}\label{soundness-gwGL}
  All the rules of \(\g{\wGL{n}} + \cut\) are sound in \(\wGL{n}\).
\end{lemma}
\begin{proof}
  Let \(R\) be one of \((\ax)\), \((\botL)\), \((\botR)\), \((\toL)\), \((\toR)\) or \((\cut)\).
  Then for any instance \((S_0,\ldots,S_{n-1},S) \in R\) we have that \(\bigwedge_{i < n} S^\flat_i \to S^\flat\) is a propositional tautology, so \(\wGL{n} \vdash \bigwedge_{i < n} S^\flat_i \to S^\flat\).
  From here it is trivial to show that the rule is sound in \(\wGL{n}\).

  Finally, assume that \(R = (\modal{\wGL{n}})\), and let us have one of its instances:
  \[
    \AxiomC{\(\necd^n \Sigma, \nec^n \phi \Rightarrow \phi\)}
    \UnaryInfC{\(\nec \Sigma, \Gamma \Rightarrow \nec \phi, \Delta\)}
    \DisplayProof
  \]
  Then we have the following reasoning.
  \begin{align*}
    &\wGL{n} \vdash \bigwedge \necd^n \Sigma \wedge \nec^n \phi \to \phi,
    &&\text{by assumption}, \\
    &\wGL{n} \vdash \bigwedge \necd^n \Sigma \to (\nec^n \phi \to \phi),
    &&\text{by propositional reasoning}, \\
    &\wGL{n} \vdash \bigwedge \nec\necd^n \Sigma \to \nec(\nec^n \phi \to \phi),
    &&\text{by \((\NEC)\) and (\(\Kax\))}, \\
    &\wGL{n} \vdash \bigwedge \nec \Sigma \to \nec(\nec^n \phi \to \phi),
    &&\text{by \(n\)-transitivity}, \\
    &\wGL{n} \vdash \bigwedge \nec \Sigma \to \nec\phi,
    &&\text{by \((\Lax[n])\)}, \\
    &\wGL{n} \vdash \bigwedge \nec \Sigma \wedge \bigwedge \Gamma \to \nec\phi \vee \bigvee \Delta,
    &&\text{by propositional reasoning}. \qedhere
  \end{align*}
\end{proof}

\begin{theorem}
  For any sequent \(S\) we have that \(\wGL{n} \vdash S^\flat\) iff \(\g{\wGL{n}} + \cut \vdash S\).
\end{theorem}
\begin{proof}
  First assume that \(\wGL{n} \vdash \bigwedge \Gamma \to \bigvee \Delta\), by Lemma~\ref{from-wGL-to-gwGL} we have that \(\g{\wGL{n}} + \cut \vdash { \Rightarrow \bigwedge \Gamma \to \bigvee \Delta}\).
  Then using invertibility we obtain that \(\g{\wGL{n}} + \cut \vdash \Gamma \Rightarrow \Delta\), as desired.
  For the if direction it suffices to use an induction on the height on the \(\g{\wGL{n}} + \cut\) proof and use the soundness of the rules (see Lemma~\ref{soundness-gwGL}).
\end{proof}

\section{A syntactic proof of cut elimination}
\label{sec:cut-elimination}

In \cite{iwata} the author asks if there is an effective cut elimination procedure for \(\g{\wGL{n}}\).
Here we will describe a syntactic cut elimination method for \(\g{\wGL{n}}\).
Even if, as it is currently described in this paper, the procedure may not look effective (as we will use non-wellfounded proofs) we will argue at Subsection~\ref{subsec:effectiveness-of-cut-elimination} that it can easily be made effective.\footnote{We have two motivations for putting the non-effective methodology (and not its effective version) as the principal one in this text. First, it can be argued to be conceptually simpler, as keeping the non-wellfounded proofs cyclic will require some (simple, but technical) additional work. Secondly, with the effective methodology the cut elimination is defined only for cyclic proofs and not for the whole class of non-wellfounded proofs.}.

The procedure we will use for cut elimination was introduced in \cite{coalgebraic} and has been succesfully applied to many provability logics \cite{proofth-interpretability, gls-ulip, bimodal-ulip}.
To be more precise, we will define a local progress sequent calculus \(\n{\wGL{n}}\).
We will show how to translate proofs from \(\g{\wGL{n}} + \cut\) to \(\n{\wGL{n}} + \cut\) and from \(\n{\wGL{n}}\) to \(\g{\wGL{n}}\).
Then we will show cut elimination in \(\n{\wGL{n}}\) and the cut elimination for \(\g{\wGL{n}}\) will follow immediately.
This procedure can be pictured as follows:
\[\begin{tikzcd}
	{\g{\wGL{n}} + \cut} && {\n{\wGL{n}} + \cut} \\
	{\g{\wGL{n}}} && {\n{\wGL{n}}}
	\arrow[from=1-1, to=1-3]
	\arrow[from=1-3, to=2-3]
	\arrow[from=2-1, to=1-1]
	\arrow[from=2-3, to=2-1]
\end{tikzcd}\]

\subsection{Non-wellfounded calculus}
\label{subsec:non-wellfounded-calculus}

The non-wellfounded calculus  \(\n{\wGL{n}}\) is defined like \(\g{\wGL{n}}\) with the modal rule replaced by the rule given in Figure~\ref{fig:rules-for-nwGL}.

\begin{figure}
  \[
  \AxiomC{\(\necd^n \Sigma \Rightarrow \phi\)}
  \RightLabel{\(\modal{\wKT{n}}\)}
  \UnaryInfC{\(\nec \Sigma, \Gamma \Rightarrow \nec \phi, \Delta\)}
  \DisplayProof
  \]
  \caption{Rules for \(\n{\wGL{n}}\)}
  \label{fig:rules-for-nwGL}
\end{figure}

\begin{definition}
  We define the local progress sequent calculus \(\n{\wGL{n}}\) as the one with the rules \((\ax)\), \((\botL)\), \((\botR)\), \((\toL)\) and \((\toR)\) from Figure~\ref{fig:rules-gwGL} and the rule \((\modal{\wKT{n}})\) from Figure~\ref{fig:rules-for-nwGL}.
  Progress is only made at the premise of \((\modal{\wKT{n}})\).
\end{definition}

Again, it is easy to show  the following lemma by induction on the complexity of \(\phi\).
\begin{lemma}
  For any formula \(\phi\) and finite multisets \(\Gamma,\Delta\) we have that \(\n{\wGL{n}} \vdash \phi, \Gamma \Rightarrow \phi, \Delta\).
\end{lemma}
When this lemma is invoked inside a proof in \(\n{\wGL{n}}\) we will annotate the sequent with the rule \((\Ax)\).
Note that this calculus also behaves well structurally, as the next lemma shows.

\begin{lemma}
  In \(\n{\wGL{n}}\) and \(\n{\wGL{n}} + \cut\) we have the following.
  \begin{enumerate}
    \item \(\wk\) is eliminable and admissible preserving local height and local rules.
    \item \(\botR\), \(\toL\) and \(\toR\) are invertible preserving local height and local rules.
    \item \(\ctr\) is eliminable and admissible preserving local height and local rules.
  \end{enumerate}
\end{lemma}
\begin{proof}
  The proof is by induction on the local height.
\end{proof}

\subsection{Translating proofs}\label{subsec:translations}

For the translation from \(\g{\wGL{n}} + \cut\) to \(\n{\wGL{n}} + \cut\) it suffices to show that L\"ob's rule is admissible in \(\g{\wGL{n}} + \cut\).

\begin{lemma}[L\"obs rule \cite{iwata}]
  The rule
  \[
    \AxiomC{\(\necd^n \Sigma, \nec^n \phi \Rightarrow \phi\)}
    \RightLabel{\(\lob\)}
    \UnaryInfC{\(\necd^n \Sigma \Rightarrow \phi\)}
    \DisplayProof
  \]
  is admissible in \(\g{\wGL{n}} + \cut\).
\end{lemma}
\begin{proof}
  Let \(\pi \vdash \necd^n \Sigma, \nec^n \phi \Rightarrow \phi\) in \(\g{\wGL{n}} + \cut\).
  Then we have the following proof
  \[
    \AxiomC{\(\pi\)}
    \noLine
    \UnaryInfC{\(\necd^n \Sigma, \nec^n \phi \Rightarrow \phi\)}
    \RightLabel{\(\modal{\wGL{n}} (\nec \Sigma \Rightarrow \nec \phi)\)}
    \UnaryInfC{\(\necd^n\nec \Sigma, \nec^{n+1}\phi \Rightarrow \nec\phi\)}
    \RightLabel{\(\modal{\wGL{n}}(\nec^{2} \Sigma \Rightarrow \nec^{2} \phi)\)}
    \UnaryInfC{\(\raisebox{0.5em}{\vdots}\)}
    \RightLabel{\(\modal{\wGL{n}}(\nec^{n-3} \Sigma \Rightarrow \nec^{n-3} \phi)\)}
    \UnaryInfC{\(\necd^n\nec^{n-2} \Sigma, \nec^{2n-2}\phi \Rightarrow \nec^{n-2}\phi\)}
    \RightLabel{\(\modal{\wGL{n}}(\nec^{n-1} \Sigma \Rightarrow \nec^{n-1} \phi)\)}
    \UnaryInfC{\(\necd^n\nec^{n-1} \Sigma, \nec^{2n-1}\phi \Rightarrow \nec^{n-1}\phi\)}
    \RightLabel{\(\modal{\wGL{n}} (\nec^n \Sigma \Rightarrow \nec^n \phi)\)}
    \UnaryInfC{\(\necd^n \Sigma \Rightarrow \phi, \nec^{n}\phi\)}
    \AxiomC{\( \pi\)}
    \noLine
    \UnaryInfC{\(\necd^n \Sigma, \nec^n \phi \Rightarrow \phi\)}
    \RightLabel{\(\cut\)}
    \BinaryInfC{\(\necd^n \Sigma \Rightarrow \phi\)}
    \DisplayProof
  \]
  where to the right of \((\modal{\wGL{n}})\) we have annotated the active part of the sequent in the rule instance.
\end{proof}

\begin{theorem}\label{from-gwGL-to-nwGL}
  \(\g{\wGL{n}} + \cut \vdash S\) implies \(\n{\wGL{n}} + \cut \vdash S\)
\end{theorem}
\begin{proof}
  We define a function \(\alpha\) from proofs in \(\g{\wGL{n}} + \cut\) to preproofs in \(\n{\wGL{n}} + \cut\) of the same sequent via corecursion.
  Once the function is defined, we will argue that \(\alpha(\pi)\) is always a proof and not only a preproof.
  Let \(\pi\) be a proof in \(\g{\wGL{n}} + \cut\) and let us defined \(\alpha(\pi)\).
  We proceed by cases in the last rule \(R\) of \(\pi\).

  Assume \(R\) is one of \((\ax)\), \((\botL)\), \((\botR)\), \((\toL)\), \((\toR)\), \((\cut)\).
  Then \(\alpha\) is defined as
  \[
    \AxiomC{\(\pi_0\)}
    \noLine
    \UnaryInfC{\(S_0\)}
    \AxiomC{\(\cdots\)}
    \AxiomC{\(\pi_{n-1}\)}
    \noLine
    \UnaryInfC{\(S_{n-1}\)}
    \RightLabel{\(R\)}
    \TrinaryInfC{\(S\)}
    \DisplayProof
    \quad
    \overset{\alpha}{\longmapsto}
    \quad
    \AxiomC{\(\alpha(\pi_0)\)}
    \noLine
    \UnaryInfC{\(S_0\)}
    \AxiomC{\(\cdots\)}
    \AxiomC{\(\alpha(\pi_{n-1})\)}
    \noLine
    \UnaryInfC{\(S_{n-1}\)}
    \RightLabel{\(R\)}
    \TrinaryInfC{\(S\)}
    \DisplayProof
  \]

  Finally, let \(R = (\modal{\wGL{n}})\), then we define \(\alpha\) as follows
  \[
    \AxiomC{\(\pi_0\)}
    \noLine
    \UnaryInfC{\(\necd^n \Sigma, \nec^n \phi \Rightarrow \phi\)}
    \RightLabel{\(\modal{\wGL{n}}\)}
    \UnaryInfC{\(\nec \Sigma, \Gamma \Rightarrow \nec \phi, \Delta\)}
    \DisplayProof
    \quad
    \longmapsto
    \quad
    \AxiomC{\(\alpha(\lob(\pi_0))\)}
    \noLine
    \UnaryInfC{\(\necd^n \Sigma  \Rightarrow \phi\)}
    \RightLabel{\(\modal{\wKT{n}}\)}
    \UnaryInfC{\(\nec \Sigma, \Gamma \Rightarrow \nec \phi, \Delta\)}
    \DisplayProof
  \]

  We check that \(\alpha(\pi)\) is indeed a proof, for that purpose let us assign the measure \(\lhg(\pi)\) to the proofs in \(\g{\wGL{n}} + \cut\).
  We notice that if \(R \neq (\modal{\wGL{n}})\) the the corecursive call has a smaller measure than the original input.
  On the other hand, if \(R = (\modal{\wGL{n}})\) then there is an application of \((\modal{\wKT{n}})\) from the root to the corecursive call in \(\alpha(\pi)\).
  Then any infinite branch of \(\alpha(\pi)\) must have infinitely many applications of \((\modal{\wKT{n}})\), as desired.
\end{proof}

Now we will show how to translate proofs from \(\n{\wGL{n}}\) to \(\g{\wGL{n}}\).
A priori, \(\n{\wGL{n}}\) does not have a nice subformula property, which could make the translation from \(\n{\wGL{n}}\) to \(\g{\wGL{n}}\) harder to establish.
However, we will prove that if we use a refined notion of subformula and we distinguish between formulas occurring at the left and at the right, then the calculus behaves well enough.

When working with pairs of sets \((\Gamma, \Delta)\), we will use the following notations:
\begin{multicols}{2}
  \begin{enumerate}
    \item \(x \in_L (\Gamma, \Delta)\) iff \(x \in \Gamma\),
    \item \(x \in_R (\Gamma, \Delta)\) iff \(x \in \Delta\),
    \item \((\Gamma, \Delta) \subseteq (\Gamma', \Delta')\) if \(\Gamma \subseteq \Gamma'\) and \(\Delta \subseteq \Delta'\).
    \item \((\Gamma, \Delta) \union (\Gamma', \Delta') = (\Gamma \union \Gamma', \Delta \union \Delta')\).
  \end{enumerate}
\end{multicols}
In the following definition we want to understand \(\psi \in_L \sub^L_{\n{\wGL{n}}}(\phi)\) as \(\psi\) can occur at the left side of a sequent if \(\phi\) occurs at the left and \(\psi \in_R \sub^L_{\n{\wGL{n}}}(\phi)\) as \(\psi\) can occur at the right side of a sequent if \(\phi\) occurs at the left. \(\sub^R_{\n{\wGL{n}}}(\phi)\) is defined similarly but with \(\phi\) occurring at the right side.

\begin{definition}
  Let \(\phi\) be a formula, we define the sets \(\sub^L(\phi)\) and \(\sub^R(\phi)\) as follows.
  \[
    \sub^L_{\n{\wGL{n}}}(p) = (\set{p}, \varnothing);\qquad\qquad\qquad
    \sub^R_{\n{\wGL{n}}}(p) = (\varnothing, \set{p});
  \]
  \[
    \sub^L_{\n{\wGL{n}}}(\bot) = (\set{\bot}, \varnothing); \qquad\qquad\qquad
    \sub^R_{\n{\wGL{n}}}(\bot) = (\varnothing, \set{\bot});
  \]
  \begin{multline*}
    \sub^L_{\n{\wGL{n}}}(\phi \to \psi) = (\set{\phi \to \psi} \union \Gamma_\phi \union \Gamma_\psi, \Delta_\phi \union \Delta_\psi),\\
\text{where \(\sub^R_{\n{\wGL{n}}}(\phi) = (\Gamma_\phi, \Delta_\phi)\) and \(\sub^L_{\n{\wGL{n}}}(\psi) = (\Gamma_\psi, \Delta_\psi)\);}
\end{multline*}
  \begin{multline*}
    \sub^R_{\n{\wGL{n}}}(\phi \to \psi) = ( \Gamma_\phi \union \Gamma_\psi, \set{\phi \to \psi} \union \Delta_\phi \union \Delta_\psi),\\
  \text{where \(\sub^L_{\n{\wGL{n}}}(\phi) = (\Gamma_\phi, \Delta_\phi)\) and \(\sub^R_{\n{\wGL{n}}}(\psi) = (\Gamma_\psi, \Delta_\psi)\);}
 \end{multline*}
  \begin{equation*}
    \sub^L_{\n{\wGL{n}}}(\nec \phi) = (\set{\nec^i \phi \mid i \geq 1} \union \Gamma_\phi, \Delta_\phi),  
  \text{where \(\sub^L_{\n{\wGL{n}}}(\phi) = (\Gamma_\phi, \Delta_\phi)\);}
 \end{equation*}
  \begin{equation*}
    \sub^R_{\n{\wGL{n}}}(\nec \phi) = (\Gamma_\phi, \set{\nec \phi} \union \Delta_\phi),  
  \text{where \(\sub^R_{\n{\wGL{n}}}(\phi) = (\Gamma_\phi, \Delta_\phi)\).}
\end{equation*}
  For a sequent \(S = \Gamma \Rightarrow \Delta\), we define \(\sub_{\n{\wGL{n}}}(S) = \Union_{\phi \in \Gamma} \sub^L(\phi) \union \Union_{\psi \in \Delta} \sub^R(\psi)\).
   If \(\sub_{\n{\wGL{n}}}(S) = (\Gamma, \Delta)\), we will write \(\sub^L_{\n{\wGL{n}}}(S)\) to denote \(\Gamma\), i.e., the formulas that can occur at the left side of the sequents; and \(\sub^R_{\n{\wGL{n}}}(S)\) to denote \(\Delta\), i.e., the formuals that can occur at the right side of the sequents.
\end{definition}

The following are trivial observations from the definition.
\begin{enumerate}
  \item For any formula \(\phi\) we have that \(\phi \in_L \sub^L_{\n{\wGL{n}}}(\phi)\) and \(\phi \in_R \sub^R_{\n{\wGL{n}}}(\phi)\).
  \item If \(\sub^L_{\n{\wGL{n}}}(\phi) = (\Gamma_\phi, \Delta_\phi)\) then \(\Delta_\phi\) is finite, similarly if \(\sub^R_{\n{\wGL{n}}}(\phi) = (\Gamma_\phi, \Delta_\phi)\) then \(\Delta_\phi\) is finite.
    As a corollary, \(\sub^R_{\n{\wGL{n}}}(S)\) is finite for any sequent \(S\).
\end{enumerate}

Then, we can show that the subformula property holds with respect to \(\sub_{\n{\wGL{n}}}\).

\begin{lemma}[Local subformula property]
  \label{local-subformula-property-nwGL}
  Let \(R\) be a rule of \(\n{\wGL{n}}\) and \((S_0,\ldots, S_{m-1}, S) \in R\).
  Then \(\sub_{\n{\wGL{n}}}(S_i) \subseteq \sub_{\n{\wGL{n}}}(S)\) for \(i < m\).
\end{lemma}
\begin{proof}
  By inspecting the shape of the rules of \(\n{\wGL{n}}\).
\end{proof}

For the translation, it is only necessary to focus on the formulas occurring at the right of the sequent.

\begin{theorem}\label{from-nwGL-to-gwGL}
  Let \((\Lambda_i)_{i < n}\) be a sequence of finite sets of formulas.
  For any sequent such that \(\n{\wGL{n}} \vdash \Gamma \Rightarrow \Delta\), we have that \(\g{\wGL{n}} \vdash \set{\nec^{i+1} \Lambda_i}_{i < n}, \Gamma \Rightarrow \Delta\).
\end{theorem}
\begin{proof}
  We use induction on the measure
  \[
    \omega\left(\sum_{i < n} |\sub^R_{\n{\wGL{n}}}(\Gamma \Rightarrow \Delta) \setminus \Lambda_i|\right) +  \lhg(\pi),
  \]
  where, for this measure to work, we need that \(\sub^R_{\n{\wGL{n}}}(\Gamma \Rightarrow \Delta)\) is always finite.
  Proceed by cases on the last rule \(R\) applied to \(\pi\).

  If \(R\) is one of \((\ax)\), \((\botL)\), \((\botR)\), \((\toL)\) or \((\toR)\) then \(\pi\) has the shape
    \[
      \AxiomC{\(\pi_0\)}
      \noLine
      \UnaryInfC{\(\Gamma_0 \Rightarrow \Delta_0\)}
      \AxiomC{\(\cdots\)}
      \AxiomC{\(\pi_{n-1}\)}
      \noLine
      \UnaryInfC{\(\Gamma_{n-1} \Rightarrow \Delta_{n-1}\)}
      \RightLabel{\(R\)}
      \TrinaryInfC{\(\Gamma \Rightarrow \Delta\)}
      \DisplayProof
    \]
    By the induction hypothesis applied to \(\pi_i\) and \((\Lambda_i)_{i < n}\) (the \(\omega\)-coefficient of the measure does not increase by the local subformula property and the other coefficient strictly decreases), we obtain  \(\tau_i \vdash \set{\nec^{i+1}\Lambda_i}_{i < n},\Gamma_i \Rightarrow \Delta_i\) for \(i < n\) proofs in \(\g{\wGL{n}}\).
    Then the desired proof is
    \[
      \AxiomC{\(\tau_0\)}
      \noLine
      \UnaryInfC{\(\set{\nec^{i+1}\Lambda_i}_{i < n},\Gamma_0 \Rightarrow \Delta_0\)}
      \AxiomC{\(\ldots\)}
      \AxiomC{\(\tau_{n-1}\)}
      \noLine
      \UnaryInfC{\(\set{\nec^{i+1}\Lambda_i}_{i < n},\Gamma_{n-1} \Rightarrow \Delta_{n-1}\)}
      \RightLabel{\(R\)}
      \TrinaryInfC{\(\set{\nec^{i+1}\Lambda_i}_{i < n},\Gamma \Rightarrow \Delta\)}
      \DisplayProof
    \]
  Now assume \(R = (\modal{\wKT{n}})\). So \(\pi\) has the following shape
  \[
    \AxiomC{\(\pi_0\)}
    \noLine
    \UnaryInfC{\(\necd^n \Sigma \Rightarrow \phi\)}
    \RightLabel{\(\modal{\wKT{n}}\)}
    \UnaryInfC{\(\nec \Sigma, \Gamma \Rightarrow \nec \phi, \Delta\)}
    \DisplayProof
  \]
  Let us denote the conclusion of \(\pi\) as \(S\) and the conclusion of \(\pi_0\) as \(S_0\).
  If \(\phi \in \Lambda_0\), then the desired proof in \(\g{\wGL{n}}\) is simply
  \[
    \AxiomC{}
    \RightLabel{\(\Ax\)}
    \UnaryInfC{\(\set{\nec^{i+1} \Lambda_i}_{i < n}, \nec \Sigma, \Gamma \Rightarrow \nec \phi, \Delta\)}
    \DisplayProof
  \]
  If \(\phi \not \in \Lambda_0\), since \(\phi \in \sub^R_{\n{\wGL{n}}}(S_0)\) we have that 
  \[
    |\sub^R_{\n{\wGL{n}}}(S_0) \setminus (\Lambda_0 \union \set{\phi})| < |\sub^R_{\n{\wGL{n}}}(S_0) \setminus \Lambda_0| \leq |\sub^R_{\n{\wGL{n}}}(S) \setminus \Lambda_0|
  \]
  where we used also that \(\sub^R_{\n{\wGL{n}}}(S_0) \subseteq \sub^R_{\n{\wGL{n}}}(S)\).
  By this inclusion, we also have that
  \[
    |\sub^R_{\n{\wGL{n}}}(S_0) \setminus \Lambda_i| \leq |\sub^R_{\n{\wGL{n}}}(S) \setminus \Lambda_i|
  \]
  for \(1 \leq i < n\).
  So we can apply the induction hypothesis to \(\pi_0\) and \((\Lambda_1,\ldots,\Lambda_{n-1}, \Lambda_0 \union \set{\phi})\) obtaining a proof  \(\tau_0 \vdash \nec \Lambda_1, \ldots, \nec^{n-1} \Lambda_{n-1}, \nec^n \Lambda_0, \nec^n \phi, \necd^n \Sigma \Rightarrow \phi\) in \(\g{\wGL{n}}\).
  The desired proof is then
  \[
    \AxiomC{\(\tau_0\)}
    \noLine
    \UnaryInfC{\(\nec \Lambda_1, \ldots, \nec^{n-1} \Lambda_{n-1}, \nec^n \Lambda_0, \nec^n \phi, \necd^n \Sigma \Rightarrow \phi\)}
    \RightLabel{\(\wk\)}
    \UnaryInfC{\(\necd^n \nec \Lambda_1, \ldots, \necd^n\nec^{n-1} \Lambda_{n-1}, \necd^n \Lambda_0, \nec^n \phi, \necd^n \Sigma \Rightarrow \phi\)}
    \RightLabel{\(\modal{\wGL{n}}\)}
    \UnaryInfC{\(\set{\nec^{i+1} \Lambda_i}_{i < n}, \nec \Sigma, \Gamma \Rightarrow \nec \phi, \Delta\)}
    \DisplayProof
  \]
  where we used that \(\wk\) is admissible in \(\g{\wGL{n}}\).
\end{proof}

\subsection{Cut elimination}
\label{subsec:cut-elimination}

We show cut elimination in \(\n{\wGL{n}}\) by showing local cut admissibility.
Then, from the translations of the previous subsection, we will automatically obtain cut elimination for \(\g{\wGL{n}}\).

\begin{theorem}\label{cut-elimination-wGL}
  \(\cut\) is eliminable in \(\n{\wGL{n}}\).
  As a corollary, \(\cut\) is eliminable in \(\g{\wGL{n}}\).
\end{theorem}
\begin{proof}
  We will show local admissibility of \(\cut\) in \(\n{\wGL{n}}\), from which the desired result follows.
  Let \(\pi \vdash \Gamma \Rightarrow \Delta, \chi\) and \(\tau \vdash \chi, \Gamma \Rightarrow \Delta\) be locally cut-free proofs in \(\n{\wGL{n}}\).
  We proceed by induction on the lexicographic pair \((|\chi|, \lhg(\pi) + \lhg(\tau))\).
  We only show one case here. The other cases are standard and can be found in Appendix~\ref{sec:cut-reductions}.
  
  Assume \(\pi\) and \(\tau\) are of the following shape:
      \[
        \AxiomC{\(\pi_0\)}
        \noLine
        \UnaryInfC{\(\necd^n \Sigma \Rightarrow \chi_0\)}
        \RightLabel{\(\modal{\wKT{n}}\)}
        \UnaryInfC{\(\nec \Sigma, \Gamma' \Rightarrow \nec \phi, \Delta', \nec\chi_0\)}
        \DisplayProof
        \qquad
        \AxiomC{\(\tau_0\)}
        \noLine
        \UnaryInfC{\(\necd^n\chi_0, \necd^n \Sigma \Rightarrow \phi\)}
        \RightLabel{\(\modal{\wKT{n}}\)}
        \UnaryInfC{\(\nec\chi_0, \nec \Sigma, \Gamma' \Rightarrow \nec \phi, \Delta'\)}
        \DisplayProof
      \]
      where \(\Gamma = \nec \Sigma, \Gamma'\), \(\Delta = \nec \phi, \Delta'\) and \(\chi = \nec \chi_0\).
      The desired proof is
      \[
        \AxiomC{\(\wk(\pi_0)\)}
        \noLine
        \UnaryInfC{\(\necd^n \Sigma \Rightarrow \phi, \chi_0\)}
        \AxiomC{\(\pi_0\)}
        \noLine
        \UnaryInfC{\(\necd^n \Sigma \Rightarrow \chi_0\)}
        \RightLabel{\(\modal{\wKT{n}} (\nec \Sigma \Rightarrow \nec \chi_0)\)}
        \UnaryInfC{\(\necd^n\nec \Sigma \Rightarrow \nec\chi_0\)}
        \RightLabel{\(\modal{\wKT{n}}(\nec^{2} \Sigma \Rightarrow \nec^{2} \chi_0)\)}
        \UnaryInfC{\(\raisebox{0.5em}{\vdots}\)}
        \RightLabel{\(\modal{\wKT{n}}(\nec^{n-3} \Sigma \Rightarrow \nec^{n-3} \chi_0)\)}
        \UnaryInfC{\(\necd^n\nec^{n-2} \Sigma \Rightarrow \nec^{n-2}\chi_0\)}
        \RightLabel{\(\modal{\wKT{n}}(\nec^{n-1} \Sigma \Rightarrow \nec^{n-1} \chi_0)\)}
        \UnaryInfC{\(\necd^n\nec^{n-1} \Sigma \Rightarrow \nec^{n-1}\chi_0\)}
        \RightLabel{\(\modal{\wKT{n}} (\nec^n \Sigma \Rightarrow \nec^n \chi_0)\)}
        \UnaryInfC{\(\chi_0,\necd^n \Sigma \Rightarrow \phi,\nec^{n}\chi_0\)}
        \AxiomC{\(\tau_0\)}
        \noLine
        \UnaryInfC{\(\necd^n\chi_0, \necd^n \Sigma \Rightarrow \phi\)}
        \RightLabel{\(\cut\)}
        \BinaryInfC{\(\chi_0, \necd^n \Sigma \Rightarrow \phi\)}
        \RightLabel{\(\cut\)}
        \BinaryInfC{\(\necd^n \Sigma \Rightarrow \phi\)}
        \RightLabel{\(\modal{\wKT{n}}\)}
        \UnaryInfC{\(\nec \Sigma, \Gamma' \Rightarrow \nec\phi, \Delta'\)}
        \DisplayProof
      \]
      where to the right of \((\modal{\wKT{n}})\) we have annotated the active part of the sequent in the rule instance.
\end{proof}

The following corollary summarizes our results.

\begin{corollary}
  For any sequent \(S\), the following are equivalent.
  \begin{enumerate}
    \item \(\wGL{n} \vdash S^\flat\),
    \item \(\g{\wGL{n}} \; (+\cut) \vdash S\),
    \item \(\n{\wGL{n}} \; (+\cut) \vdash S\).
  \end{enumerate}
\end{corollary}

\subsection{Effectiveness of cut elimination}
\label{subsec:effectiveness-of-cut-elimination}

In this subsection we will discuss the effectiveness of the cut elimination for \(\g{\wGL{n}}\) that we provided at Subsections~\ref{subsec:translations} and \ref{subsec:cut-elimination}.
We will show that, instead of non-wellfounded proofs, the use of cyclic proofs suffices and thus all the steps can be effectively computed, thus answering a question from \cite{iwata}.

The cut elimination procedure for \(\g{\wGL{n}}\) presented in the previous section consists of  three steps:
\begin{enumerate}
  \item Translate the proof from \(\g{\wGL{n}} + \cut\) to \(\n{\wGL{n}} + \cut\), corecursively.
  \item Eliminate the cuts in \(\n{\wGL{n}} + \cut\) using corecursion on the local fragments.
  \item Translate the proof from \(\n{\wGL{n}}\) to \(\g{\wGL{n}}\) by corecursion on the local fragments.
\end{enumerate}
In this section, we will restructure this procedure as follows.
\begin{enumerate}
  \item Translate the proof from \(\g{\wGL{n}} + \cut\) to a cyclic proof in \(\n{\wGL{n}} + \cut\), recursively.
  \item Translate from cyclic proofs in \(\n{\wGL{n}} + \cut\) to \(\g{\wGL{n}}\) by induction.
    For this we will need an inductively defined function that from a cyclic proof in \(\n{\wGL{n}} + \cut\) produces a cyclic locally cut free proof in \(\n{\wGL{n}} + \cut\).
\end{enumerate}
In some sense, the second step of this procedure is obtained from merging the last two steps of the previously defined procedure.
This can be understood as follows: to obtain a proof in \(\g{\wGL{n}}\) from a proof in \(\n{\wGL{n}}\) it is not necessary to have the full proof in \(\n{\wGL{n}}\) but just a big enough fragment of it (as the non-wellfounded behaviour of the proofs in \(\n{\wGL{n}}\) is relatively simple).
For that reason, we never have to  calculate the full cut free proof of \(\n{\wGL{n}}\).

In order to simplify the proofs a bit, we will use the concept of locally backedge-free cyclic proof and unfolding.
For the rest of the subsection we make the following convention.
\begin{center}
  Any cyclic proof given in this subsection will be assumed to be locally backedge free.
\end{center}
We can assume this without loss of generality thanks to the unfolding technique (see the last paragraph of Subsection~\ref{subsec:local-progress-proof-theory}).

Before we start, let us introduce a restriction of the contraction rule to variables as follows:
\[
  \AxiomC{\(\Gamma, \Gamma' \Rightarrow \Delta, \Delta'\)}
  \RightLabel{\(\ctr_{\var}\)}
  \UnaryInfC{\(\Gamma, (\Gamma')^s \Rightarrow \Delta, (\Delta')^s\)}
  \DisplayProof
\]
where \(\Gamma', \Delta' \subseteq \var\).\footnote{The reason we restrict the contraction rule is that it simplifies the proof of cyclicity preserving admissibility.}

\begin{lemma}
  We have the following in \(\n{\wGL{n}}\) and in \(\n{\wGL{n}} + \cut\).
  \begin{enumerate}
    \item \(\wk\) is  cyclic admissible preserving local height and local rules.
    \item \((\botR)\), \((\toL)\) and \((\toR)\) are cyclic invertible preserving local height and local rules.
    \item \((\ctr_{\mathrm{Var}})\) is  cyclic admissible preserving local height and local rules.
  \end{enumerate}
\end{lemma}
\begin{proof}
  The proof is a simple induction in the local height of the proofs.
\end{proof}

First, we show that there is a cyclic proof of L\"ob's axiom in \(\n{\wGL{n}}\).

\begin{lemma}\label{lob-axiom-cyclic-proof}
  For any \(\Gamma, \Delta\) there is a cyclic proof in \(\n{\wGL{n}}\) of \(\nec(\nec^n \phi \to \phi), \Gamma \Rightarrow \nec \phi, \Delta\).
\end{lemma}
\begin{proof}
  We have the following proof
  \[
    \AxiomC{}
    \RightLabel{\(\Ax\)}
    \UnaryInfC{\(\phi, \nec^n(\nec^n \phi \to \phi) \Rightarrow  \phi\)}
    \AxiomC{}
    \RightLabel{\(\Ax\)}
    \UnaryInfC{\(\phi,\nec^{n}(\nec^n \phi \to \phi) \Rightarrow \phi\)}
    \AxiomC{\colorbox{gray!40}{\(\nec^{n}(\nec^n \phi \to \phi) \Rightarrow \nec^n \phi, \phi\)}}
    \RightLabel{\(\toL\)}
    \BinaryInfC{\(\necd^{n}(\nec^n \phi \to \phi) \Rightarrow \phi\)}
    \RightLabel{\(\modal{\wKT{n}}\)}
    \UnaryInfC{\(\necd^{n}\nec(\nec^n \phi \to \phi) \Rightarrow \nec \phi\)}
    \RightLabel{\(\modal{\wKT{n}}\)}
    \UnaryInfC{\(\raisebox{0.5em}{\vdots}\)}
    \UnaryInfC{\(\necd^{n}\nec^{n-1}(\nec^n \phi \to \phi) \Rightarrow \nec^{n-1} \phi\)}
    \RightLabel{\(\modal{\wKT{n}}\)}
    \UnaryInfC{\colorbox{gray!40}{\(\nec^n(\nec^n \phi \to \phi) \Rightarrow \nec^n \phi, \phi\)}}
    \RightLabel{\(\toL\)}
    \BinaryInfC{\( \necd^n(\nec^n \phi \to \phi) \Rightarrow  \phi\)}
    \RightLabel{\(\modal{\wKT{n}}\)}
    \UnaryInfC{\( \nec(\nec^n \phi \to \phi), \Gamma \Rightarrow  \nec \phi, \Delta\)}
    \DisplayProof
  \]
  where the cycle is marked by the graybox.
  We used that any proof generated by \((\Ax)\) is finite.
\end{proof}

\begin{lemma}\label{effective-elim-first-step}
  For any proof \(\pi \vdash S\) in \(\g{\wGL{n}} + \cut\) there is a cyclic proof \(\pi^\circ \vdash S\) in \(\n{\wGL{n}} + \cut\).
  In addition, the function \(\pi \mapsto \pi^\circ\) is effective.
\end{lemma}
\begin{proof}
  We define \(\pi^\circ\) by recursion on the height of \(\pi\) and cases on the last rule \(R\) of \(\pi\).
  If \(R\) is one of \((\ax)\), \((\botL)\), \((\botR)\), \((\toL)\), \((\toR)\) or \((\cut)\) then the function is defined as
  \[
    \AxiomC{\(\pi_0\)}
    \noLine
    \UnaryInfC{\(S_0\)}
    \AxiomC{\(\ldots\)}
    \AxiomC{\(\pi_{n-1}\)}
    \noLine
    \UnaryInfC{\(S_{n-1}\)}
    \RightLabel{\(R\)}
    \TrinaryInfC{\(S\)}
    \DisplayProof
    \quad
    \overset{(\cdot)^\circ}{\longmapsto}
    \quad
    \AxiomC{\(\pi^\circ_0\)}
    \noLine
    \UnaryInfC{\(S_0\)}
    \AxiomC{\(\ldots\)}
    \AxiomC{\(\pi^\circ_{n-1}\)}
    \noLine
    \UnaryInfC{\(S_{n-1}\)}
    \RightLabel{\(R\)}
    \TrinaryInfC{\(S\)}
    \DisplayProof
  \]
  If \(R\) is \(\modal{\wGL{n}}\) then the function is defined as
  \begin{multline*}
    \AxiomC{\(\pi_0\)}
    \noLine
    \UnaryInfC{\(\necd^n \Sigma, \nec^n \phi \Rightarrow \phi\)}
    \RightLabel{\(\modal{\wGL{n}}\)}
    \UnaryInfC{\(\nec \Sigma, \Gamma \Rightarrow \nec \phi, \Delta\)}
    \DisplayProof
    \quad
    \overset{(\cdot)^\circ}{\longmapsto} \\
    \AxiomC{\(\pi^\circ_0\)}
    \noLine
    \UnaryInfC{\(\necd^n \Sigma, \nec^n \phi \Rightarrow \phi\)}
    \RightLabel{\(\toR\)}
    \UnaryInfC{\(\necd^n \Sigma \Rightarrow \nec^n \phi \to \phi\)}
    \RightLabel{\(\modal{\wKT{n}}\)}
    \UnaryInfC{\(\nec \Sigma, \Gamma \Rightarrow \nec(\nec^n \phi \to \phi), \nec \phi, \Delta\)}
    \AxiomC{}
    \RightLabel{Lm.~\ref{lob-axiom-cyclic-proof}}
    \UnaryInfC{\( \nec(\nec^n \phi \to \phi), \nec \Sigma, \Gamma \Rightarrow \nec \phi, \Delta\)}
    \RightLabel{\(\cut\)}
    \BinaryInfC{\(\nec \Sigma, \Gamma \Rightarrow \nec \phi, \Delta\)}
    \DisplayProof
  \end{multline*}
\end{proof}

With this we have the first step of the procedure covered.
Now, we proceed to define the second step.

\begin{lemma}\label{local-cut-elim-cyclic}
  For any cyclic proof \(\pi \vdash \Gamma \Rightarrow \Delta\) in \(\n{\wGL{n}} + \cut\), there is a locally cut free cyclic proof \(\pi \setminus \localCut \vdash \Gamma \Rightarrow \Delta\) in \(\n{\wGL{n}} + \cut\).
  In addition, the function \(\pi \mapsto \pi \setminus \localCut\) is effective.
\end{lemma}
\begin{proof}
  The proof is by induction on the number of cuts in the main local fragment of \(\pi\), first showing that \(\cut\) is cyclic locally admissible in \(\n{\wGL{n}}\).
  This procedure is completely analogous to the proof of Theorem~\ref{cut-elimination-wGL}, so we skip the details.
\end{proof}

\begin{lemma}\label{effective-elim-second-step}
  For any locally cut free cyclic proof \(\pi \vdash \Gamma \Rightarrow \Delta\) in \(\n{\wGL{n}} + \cut\), there is a proof \(\pi^\bullet \vdash \Gamma \Rightarrow \Delta\) in \(\g{\wGL{n}}\).
\end{lemma}
\begin{proof}
  We will show that for any finite sets \((\Lambda_i)_{i < n}\) and locally cut free cyclic proof \(\pi \vdash \Gamma \Rightarrow \Delta\) there is a proof \(\pi^\bullet_{(\Lambda_i)_{i < n}} \vdash \set{\nec^{n+1} \Lambda_i}_{i < n}, \Gamma \Rightarrow \Delta\) in \(\g{\wGL{n}}\).
  From the proof, it will be clear that the process is effective.

  We proceed by induction on the measure 
  \[
    \omega\left(\sum_{i < n}|\sub^R_{\n{\wGL{n}}}(\Gamma \Rightarrow \Delta) \setminus \Lambda_i|\right) + \lhg(\pi),
  \]
  and cases on the last rule \(R\) of \(\pi\).
  If \(R\) is one of \((\ax)\), \((\botL)\), \((\botR)\), \((\toL)\) or \((\toR)\) then the proof is straightforward using the induction hypothesis in the immediate subproofs  observing that the \(\omega\)-coefficient of the measure does not increase by Lemma~\ref{local-subformula-property-nwGL} and the local height of the proof strictly decreases.
  Note that \(R\) cannot be \((\cut)\), as \(\pi\) is locally cut-free.
  Finally, assume that \(R = (\modal{\wKT{n}})\), so \(\pi\) is of shape
  \[
    \AxiomC{\(\pi_0\)}
    \noLine
    \UnaryInfC{\(\necd^n \Sigma \Rightarrow \phi\)}
    \RightLabel{\(\modal{\wKT{n}}\)}
    \UnaryInfC{\(\nec \Sigma, \Gamma \Rightarrow \nec \phi, \Delta\)}
    \DisplayProof
  \]
  where \(\pi_0\) is a cyclic proof, which we unfold if necessary, so it is in the locally backedge free respresentation.
  Let us denote the conclusion of \(\pi\) as \(S\) and the conclusion of \(\pi_0\) as \(S_0\).
  
  For defining \(\pi^\bullet_{(\Lambda_i)_{i < n}}\) we first assume that \(\phi \in \Lambda_0\).
  Then the desired proof in \(\g{\wGL{n}}\) is simply
  \[
    \AxiomC{}
    \RightLabel{\(\Ax\)}
    \UnaryInfC{\(\set{\nec^{i+1} \Lambda_i}_{i < n}, \nec \Sigma, \Gamma \Rightarrow \nec \phi, \Delta\)}
    \DisplayProof
  \]
  Now, assume that \(\phi \not \in \Lambda_0\).
  Note that \(\localCut(\pi_0) \vdash S_0\) and is a locally cut free cyclic proof in \(\n{\wGL{n}} + \cut\).
  In addition \(\phi \in \sub^R_{\n{\wGL{n}}}(S_0)\) and \(\phi \not \in \Lambda_0\) so
  \[
    |\sub^R_{\n{\wGL{n}}}(S_0) \setminus (\Lambda_0 \union \set{\phi})| < |\sub^R_{\n{\wGL{n}}}(S_0) \setminus \Lambda_0| \leq |\sub^R_{\n{\wGL{n}}}(S) \setminus \Lambda_0|
  \]
  where we used also that \(\sub^R_{\n{\wGL{n}}}(S_0) \subseteq \sub^R_{\n{\wGL{n}}}(S)\).
  By this inclusion we also have that
  \[
    |\sub^R_{\n{\wGL{n}}}(S_0) \setminus \Lambda_i| \leq |\sub^R_{\n{\wGL{n}}}(S) \setminus \Lambda_i|
  \]
  for \(1 \leq i < n\).
  Then, we can use that the inductive measure is decreased to define \(\pi^\bullet_{(\Lambda_i)_{i < n}}\) as follows
  \[
    \AxiomC{\((\pi_0 \setminus \localCut)^\bullet_{(\Lambda_1,\ldots, \Lambda_{n-1}, \Lambda_0 \union \set{\phi})}\)}
    \noLine    
    \UnaryInfC{\(\nec^n \Lambda_0, \set{\nec^i \Lambda_i}_{1 \leq i < n}, \nec^n \phi, \necd^n \Sigma \Rightarrow \phi\)}
    \RightLabel{\(\wk\)}
    \UnaryInfC{\(\necd^n \Lambda_0, \set{\necd^n \nec^i \Lambda_i}_{1 \leq i < n}, \nec^n \phi, \necd^n \Sigma \Rightarrow \phi\)}
    \RightLabel{\(\modal{\wGL{n}}\)}
    \UnaryInfC{\(\set{\necd^n \nec^{i+1} \Lambda_i}_{i < n}, \nec \Sigma, \Gamma \Rightarrow \nec \phi, \Delta\)}
    \DisplayProof
  \]
  where we used that \(\wk\) is admissible in \(\g{\wGL{n}}\).\footnote{Note that in the local height of \(\pi_0 \setminus \localCut\) might be bigger than the local height of \(\pi_0\). This does not matter us as the inductive measure is still reduced.}
\end{proof}

\begin{theorem}
  Cut is effective eliminable in \(\g{\wGL{n}}\).
\end{theorem}
\begin{proof}
  Just compose \((\cdot)^\circ\) from Lemma~\ref{effective-elim-first-step} with \(\cdot \setminus \localCut\) from Lemma~\ref{local-cut-elim-cyclic} and then with \((\cdot)^\bullet\) from Lemma~\ref{effective-elim-second-step}.
\end{proof}

\subsection{Alternative non-wellfounded system}

We now present  an alternative non-wellfounded system, which has a better subformula property. It will be needed in the next section to obtain our interpolation results.
Define the \emph{modal modulus} of a formula \(\phi\), denoted \(\mmod{\phi} \in \mathbb{Z}/\mathbb{Z}n\), as follows:
\begin{align*}
  &\mmod{p} = \mmod{\bot} = \mmod{\phi_0 \to \phi_1} = [0]_n,
  &&\mmod{\nec \phi_0} = \mmod{\phi_0} + [1]_n.
\end{align*}

The new rule of the sequent calculus is displayed at Figure~\ref{fig:alternative-rule}.
The idea is that, as \(\wGL{n} \vdash \nec \phi \to \nec^{n+1} \phi\), whenever we put \(\necd^{n} \nec \phi\) at the left side of a sequent, we are adding redundant information, and \(\nec \phi\) should suffice.

\begin{figure}
  \[
    \AxiomC{\(\necd^n \Sigma_0, \set{\nec^i \Sigma_i}_{1 \leq i < n} \Rightarrow \phi\)}
    \RightLabel{\(\modalalt{\wKT{n}}\)}
    \UnaryInfC{\(\set{\nec^{i+1} \Sigma_i}_{i < n}, \Gamma \Rightarrow \nec\phi, \Delta\)}
    \DisplayProof
  \]
  \begin{center}
  where \(\mmod{\Sigma_i} = 0\) for \(i < n\).
  \end{center}
  \caption{Modulus rules}
  \label{fig:alternative-rule}
\end{figure}

\begin{definition}
  We define the local progress sequent calculus \(\nmod{\wGL{n}}\) is the one given by the rules \((\ax)\), \((\botL)\), \((\botR)\), \((\toL)\), \((\toR)\) of Figure~\ref{fig:rules-gwGL} and the rule \((\modalalt{\wKT{n}})\) of Figure~\ref{fig:alternative-rule}.
      Progress is only made at the premise of \((\modalalt{\wKT{n}})\).
\end{definition}

As usual, it is straightforward to see that the calculus behaves well structurally.

\begin{lemma}
  We have the following in \(\nmod{\wGL{n}} \; (+ \cut)\).
  \begin{enumerate}
    \item \(\wk\) is eliminable and admissible preserving local height and local rules.
    \item \((\botR)\), \((\toL)\), \((\toR)\) are invertible preserving local height and local rules.
    \item \(\ctr\) is eliminable and admissible preserving local height and local rules.
  \end{enumerate}
\end{lemma}

Clearly, if we are looking for the local subformula property this rule does not suffice: it is not true that  for any rule instance \((S_0,S)\) of \((\modalalt{\wGL{n}})\) we have that \(\sub(S_0) \subseteq \sub(S)\).
However, we can get a nice global subformula property as the following lemma shows.

\begin{lemma}\label{global-subformula-property-nmodwGL}
  For any sequent \(S'\), rule \(R\) of \(\nmod{\wGL{n}}\) and instance \((S_0,\ldots, S_{m-1}, S)\) of \(R\) we have that \(S \subseteq \Union_{j \leq n} \nec^j \set{\chi \in \sub(S') \mid \mmod{\chi} = [0]_n}\) implies \(S_i \subseteq \Union_{j \leq n} \nec^j \set{\chi \in \sub(S') \mid \mmod{\chi} = [0]_n}\) for \(i < m\).
\end{lemma}
\begin{proof}
  If \(R\) is \((\ax)\) or \((\botL)\) then the proof is trivial.

  Assume \(R\) is \((\botR)\), \((\toL)\) or \((\toR)\).
  For any sequent \(S\), \(S \subseteq  \Union_{j \leq n} \nec^j \set{\chi \in \sub(S') \mid \mmod{\chi} = [0]_n}\) implies \(\sub(S) \subseteq  \Union_{j \leq n} \nec^j \set{\chi \in \sub(S') \mid \mmod{\chi} = [0]_n}\).
  Then this case follows since \(\sub(S_i) \subseteq \sub(S)\) for instances of \((\botR)\), \((\toL)\) and \((\toR)\).

  Finally, assume \(R\) is \((\modalalt{\wKT{n}})\).
  Then \(R\) has shape
  \[
  \AxiomC{\(\necd^n \Sigma_0, \set{\nec^{i} \Sigma_i}_{1 \leq i < n} \Rightarrow \phi\)}
\UnaryInfC{\(\set{\nec^{i+1} \Sigma_i}_{i < n}, \Gamma \Rightarrow \nec \phi, \Delta\)}
    \DisplayProof
  \]
  where \(S\) is the conlcusion and \(S_0\) the premise.
  Since \(S \subseteq  \Union_{j \leq n} \nec^j \set{\chi \in \sub(S') \mid \mmod{\chi} = [0]_n}\) we also have that \(\sub(S) \subseteq  \Union_{j \leq n} \nec^j \set{\chi \in \sub(S') \mid \mmod{\chi} = [0]_n}\), and then it is easy to show that \(\Sigma_0, {\nec^i \Sigma_i}_{1 \leq i < n} \union \set{\phi} \subseteq  \Union_{j \leq n} \nec^j \set{\chi \in \sub(S') \mid \mmod{\chi} = [0]_n}\).
  All left is \(\nec^n \Sigma_0 \subseteq \Union_{j \leq n} \nec^j \set{\chi \in \sub(S') \mid \mmod{\chi} = [0]_n}\), so let \(\psi \in \Sigma_0\).
  As \(\nec \psi\) occurs in the left side of \(S\), we have \(\nec \psi \in \Union_{j \leq n} \nec^j \set{\chi \in \sub(S') \mid \mmod{\chi} = [0]_n}\) and since \(\mmod{\nec \psi} = [1]_n\) we get \(\nec\psi \in \nec \set{\chi \in \sub(S') \mid \mmod{\chi} = [0]_n}\) so \(\psi \in  \set{\chi \in \sub(S') \mid \mmod{\chi} = [0]_n}\).
  Then it follows that \(\nec^n \psi \in \Union_{j \leq n} \nec^j \set{\chi \in \sub(S') \mid \mmod{\chi} = [0]_n}\), as desired.
\end{proof}

Additionally, we check that the \(\nmod{\wGL{n}}\) calculus is equivalent to \(\wGL{n}\) and has cut elimination.
The easiest way of doing it is by translating proofs between \(\n{\wGL{n}}\) and \(\nmod{\wGL{n}}\).

\begin{lemma}
  The rule
  \[
    \AxiomC{\(\necd^n\nec \phi, \Gamma \Rightarrow \Delta\)}
    \RightLabel{\(\wMax[n]\)}
    \UnaryInfC{\(\nec \phi, \Gamma \Rightarrow \Delta\)}
    \DisplayProof
  \]
  is admissible in \(\n{\wGL{n}}\).
\end{lemma}
\begin{proof}
  This follows trivially using cut elimination and that \(\n{\wGL{n}} \vdash \nec \phi, \Gamma \Rightarrow \nec^{n+1} \phi, \Delta\).
\end{proof}

\begin{lemma}
  For any sequent \(S\) we have that \(\n{\wGL{n}} \; (+\cut) \vdash S\) iff \(\nmod{\wGL{n}} \; (+\cut) \vdash S\).
\end{lemma}
\begin{proof}
  It suffices to provide translations from \(\nmod{\wGL{n}} + \cut\) to \(\n{\wGL{n}} + \cut\) and from \(\n{\wGL{n}}\) to \(\nmod{\wGL{n}}\).
  
  Translation from \(\nmod{\wGL{n}} + \cut\) to \(\n{\wGL{n}} + \cut\).
  We define a corecursive function \(\alpha\) from proofs in \(\nmod{\wGL{n}} + \cut\) to preproofs in \(\n{\wGL{n}} + \cut\).
  Once it is defined we will argue that \(\alpha(\pi)\) is always a proof and not only a preproof.
  We proceed by cases on the last rule \(R\) of \(\pi\), if it is one of \((\ax)\), \((\botL)\), \((\botR)\), \((\toL)\), \((\toR)\) or \((\cut)\) then it simply commutes.
  If \(R = (\modalalt{\wKT{n}})\), then we define \(\alpha\) as follows.
  \[
    \AxiomC{\(\pi_0\)}
    \noLine
    \UnaryInfC{\( \necd^n \Sigma_0, \set{\nec^{i} \Sigma_i}_{1 \leq i < n} \Rightarrow \phi\)}
    \RightLabel{\(\modalalt{\wKT{n}}\)}
    \UnaryInfC{\( \set{\nec^{i+1} \Sigma_i}_{i < n}, \Gamma \Rightarrow \nec \phi, \Delta\)}
    \DisplayProof
    \quad
    \overset{\alpha}{\longmapsto}
    \quad
    \AxiomC{\(\alpha(\wk(\pi_0))\)}
    \noLine
    \UnaryInfC{\( \necd^n \Sigma_0, \set{\necd^n\nec^{i} \Sigma_i}_{1 \leq i < n} \Rightarrow \phi\)}
    \RightLabel{\(\modal{\wKT{n}}\)}
    \UnaryInfC{\( \set{\nec^{i+1} \Sigma_i}_{i < n}, \Gamma \Rightarrow \nec\phi, \Delta\)}
    \DisplayProof
  \]
  We can show that \(\alpha(\pi)\) is a proof by using the measure \(\lhg(\pi)\) on the corecursive calls.

  Translation from \(\n{\wGL{n}}\) to \(\nmod{\wGL{n}}\).
  We define a corecursive function \(\beta\) from proofs in \(\n{\wGL{n}}\) to preproofs in \(\nmod{\wGL{n}}\).
  Once it is defined we will argue that \(\beta(\pi)\) is always a proof and not only a preproof.
  We proceed by cases on the last rule \(R\) of \(\pi\), if it is one of \((\ax)\), \((\botL)\), \((\botR)\), \((\toL)\) or \((\toR)\) then it simply commutes.
  If \(R = (\modalalt{\wKT{n}})\), then we define \(\alpha\) as follows.
  \[
    \AxiomC{\(\pi_0\)}
    \noLine
    \UnaryInfC{\(\set{\necd^{n}\nec^{i} \Sigma_i}_{i < n} \Rightarrow \phi\)}
    \RightLabel{\(\modal{\wKT{n}}\)}
    \UnaryInfC{\(\set{\nec^{i+1} \Sigma_i}_{i < n}, \Gamma \Rightarrow \nec \phi, \Delta\)}
    \DisplayProof
    \quad
    \overset{\beta}{\longmapsto}
    \quad
    \AxiomC{\(\beta(\wMax[n](\pi_0))\)}
    \noLine
    \UnaryInfC{\(\necd^{n} \Sigma_0, \set{\nec^{i} \Sigma_i}_{i < n} \Rightarrow \phi\)}
    \RightLabel{\(\modalalt{\wKT{n}}\)}
    \UnaryInfC{\(\set{\nec^{i+1} \Sigma_i}_{i < n}, \Gamma \Rightarrow \nec \phi, \Delta\)}
    \DisplayProof
  \]
  where we assumed that \(\mmod{\Sigma_i} = 0\).
  Again, we can show that \(\alpha(\pi)\) is a proof by using the measure \(\lhg(\pi)\) on the corecursive calls.
\end{proof}

Thanks to the better behaved subformula property, we immediately obtain that this proof system has terminating proof search.
For that reason, we will use its rules in the next section to create interpolation templates.

\section{Uniform Lyndon Interpolation}
\label{sec:interpolation}

Using the non-wellfounded calculi \(\n{\wGL{n}}\) and \(\nmod{\wGL{n}}\) we show that \(\wGL{n}\) has uniform Lyndon interpolation.
We will use \(\nmod{\wGL{n}}\) to create cyclic proof search trees.
Such a tree carries all the information of how the sequent can be use in any proof of \(\nmod{\wGL{n}}\).
Thanks to this, it is possible to define the uniform interpolant from it via recursion on the tree structure and solving a modal equational system to eliminate the cycles.

The use of a non-wellfounded calculus is not a caprice.
Since we are interested in the polarities of variables, as we want to show uniform Lyndon inteporlation, we need a calculus that treat the polarities adequately.
The wellfounded calculi for \(\wGL{n}\) does not fulfill this condition, as the diagonal formula makes a negative occurence of a formula that may have only positive occurences prior to its application.

\subsection{Preinterpolant}

For the interpolation it is important that we create an equational system of depth \(0\).
With this purpose in mind, we will start recording the number of applications of modal rules that we have used from the root, modulo \(n\).
A sequent will then be a triple \((\Gamma, \Delta, [i]_n)\) such that \(\Gamma, \Delta\) are finite multisets of formulas and \(i \in \mathbb{N}\).
Such a sequent will usually be denoted as \(\Gamma \Rightarrow_i \Delta\).
A \(0\)-sequent is a sequent of the shape \(\Gamma \Rightarrow_{[0]_n} \Delta\).
In Figure~\ref{fig:interpolation-template} we have the rules that we will use in interpolation templates.

\begin{figure}
  \[
    \AxiomC{\(\)}
    \RightLabel{\(\ax\)}
    \UnaryInfC{\(p, \Gamma \Rightarrow_{[i]_n} p, \Delta\)}
    \DisplayProof
    \qquad
    \AxiomC{}
    \RightLabel{\(\emp\)}
    \UnaryInfC{\( \Rightarrow_{[i]_n} \)}
    \DisplayProof
  \]

  \[
    \AxiomC{\(\)}
    \RightLabel{\(\botL\)}
    \UnaryInfC{\(\bot, \Gamma \Rightarrow_{[i]_n} \Delta\)}
    \DisplayProof
    \qquad
    \AxiomC{\(\Gamma \Rightarrow_{[i]_n} \Delta\)}
    \RightLabel{\(\botR\)}
    \UnaryInfC{\(\Gamma \Rightarrow_{[i]_n} \bot, \Delta\)}
    \DisplayProof
  \]

  \[
    \AxiomC{\(\Gamma \Rightarrow_{[i]_n} \phi, \Delta\)}
    \AxiomC{\(\psi, \Gamma \Rightarrow_{[i]_n} \Delta\)}
    \RightLabel{\(\toL\)}
    \BinaryInfC{\(\phi \to \psi, \Gamma \Rightarrow_{[i]_n} \Delta\)}
    \DisplayProof
    \qquad
    \AxiomC{\(\phi, \Gamma \Rightarrow_{[i]_n} \psi, \Delta\)}
    \RightLabel{\(\toR\)}
    \UnaryInfC{\(\Gamma \Rightarrow_{[i]_n} \phi \to \psi, \Delta\)}
    \DisplayProof
  \]

  \[
    \AxiomC{\(\necd^n \Sigma^s_0, \set{\nec^{j} \Sigma^s_j}_{1 \leq j < n} \Rightarrow_{[i]_n} \)}
    \AxiomC{\(\left[\necd^n \Sigma^s_0, \set{\nec^{j} \Sigma^s_i}_{1 \leq j < n} \Rightarrow_{[i]_n} \phi\right]_{\phi \in \Theta}\)}
    \RightLabel{\(\modal[+]{\wKT{n}}\)}
    \BinaryInfC{\(\set{\nec^{j+1} \Sigma_j}_{j < n}, \Gamma' \Rightarrow_{[i+1]_n} \nec \Theta, \Delta'\)}
    \DisplayProof
  \]
  where \(\mmod{\Sigma_i} \subseteq \set{0}\) for \(i < n\) and \(\Gamma', \Delta' \subseteq \text{Var}\) such that \(\Gamma' \inter \Delta' = \varnothing\).
  \caption{Interpolation template rules}
  \label{fig:interpolation-template}
\end{figure}

\begin{definition}
  An \emph{interpolation template for \(\Gamma \Rightarrow \Delta\)} is a cyclic proof of \(\Gamma \Rightarrow \Delta\) in the local progress calculus generated by the rules of Figure~\ref{fig:interpolation-template} where progress is only made at the premises of \((\modal[+]{\wKT{n}})\)
  In addition, we impose that any repetition node is a \(0\)-sequent.
\end{definition}

The first step is showing that any sequent has an interpolation template.
For that purpose, we first establish that the rules of the interpolation template have a subformula property.

\begin{lemma}[Global subformula property]
  For any sequent \(S'\), rule \(R\) of interpolation templates and rule instance \((S_0,\ldots, S_{m-1}, S) \in R\), we have that
  \[S \subseteq \Union_{j \leq n} \nec^j \set{\chi \in \sub(S) \mid \mmod{\chi} = [0]_n} \quad \text{implies}\quad S_i \subseteq \Union_{j \leq n} \nec^j \set{\chi \in \sub(S) \mid \mmod{\chi} = [0]_n}\] for \(i < m\).
\end{lemma}
\begin{proof}
  The proof is completely analogous to the proof of Lemma~\ref{global-subformula-property-nmodwGL}.
\end{proof}

Let \(\Gamma \Rightarrow \Delta\) be a sequent, define \(|\Gamma \Rightarrow \Delta| = \sum_{\phi \in \Gamma \union \Delta} |\phi|\), where we remember \(|\phi|\) is the logical complexity of \(\phi\).
We will use that measure in the following lemma.

\begin{lemma}
  Every sequent has an interpolation template.
\end{lemma}
\begin{proof}
  Notice that if \(R\) is distinct from \((\modal[+]{\wKT{n}})\) then \(|S_i| < |S|\) for \(i < n\).
  Also, notice that for any sequent \(\Gamma \Rightarrow_{[i]_n} \Delta\) there is always one rule of interpolation templates that can be applied to it.
  Then we proceed as follows to construct an interpolation template for \(\Gamma \Rightarrow \Delta\):
  \begin{enumerate}
    \item Stage 0. Start by putting \(\Gamma \Rightarrow_{[0]_n} \Delta\) at the root of the tree without a rule.
    \item Stage \(n+1\). For each leaf \(w\) at stage \(n\) without a rule do one of the following (in order):
      \begin{enumerate}
        \item If the sequent at \(w\) is a \(0\)-sequent and there is a \(v < w\) with the same sequent, make \(w\) a repeat node pointing to \(v\).
        \item Otherwise, select a rule \(R\) and a rule instance \((S_0,\ldots,S_{n-1},S) \in R\) such that \(w\) is annotated with sequent \(S\). Annotate \(w\) with the rule \(R\) and create new nodes \(w0\), \ldots, \(w(n-1)\) immediate successors of \(w\) annotated with sequent \(S_0,\ldots,S_{n-1}\) respectively.
      \end{enumerate}
  \end{enumerate}
  First, we show that this process finishes.

  Suppose otherwise, then at the limit we would have an non-wellfounded tree with at least one infinite branch \((w_i)_{i \in \mathbb{N}}\).
  Let \(w_i\) be annotated with sequent \(S_i\) and rule \(R_i\), notice that all the \(S_i\)s which are \(0\)-sequents are pairwise different, as otherwise the infinite branch would have been closed during the construction with a repeat.
  If there are only finitely many \(i\)s such that \(R_i = (\modal[+]{\wKT{n}})\) then there is a maximum \(i_0\) such that for any \(j > i_0\) we have that \(R_i \neq (\modal[+]{\wKT{n}})\).
  Then \(|S_{i_0}| > |S_{i_0 + 1}| > \cdots > |S_{i_0 + k}| > \cdots\), which is impossible.
  So there are infinitely many \(i\)s such that \(R_i = (\modal[+]{\wKT{n}})\).
  Additionally, as the only rule that change the number of the sequent arrow is \((\modal[+]{\wKT{n}})\) and it only changes in a cycle as \([0]_n\), \([n-1]_n\), \ldots, \([1]_n\), \([0]_n\), \ldots  
  Then it must be the case that there are infintely many \(i\)s such that \(R_i = (\modal[+]{\wKT{n}})\) and \(S_i\) is a \(0\)-sequent.
  Let \(\set{i_j}_{j \in \mathbb{N}}\) be those \(i\)s.

  Let \(m\) be the cardinality of \(\set{\phi \in \sub(\Gamma \Rightarrow \Delta) \mid \mmod{\phi} = [0]_n}\), then we have that the cardinality of the set \(\Union_{i \leq n} \nec^i \set{\phi \in \sub(\Gamma \Rightarrow \Delta) \mid \mmod{\phi} = [0]_n}\) is at most \((n+1)m\). 
  By the global subformula property and since premises of \((\modal[+]{\wKT{n}})\) are determined by \(n\)-many sets of formulas in \(\set{\phi \in \sub(\Gamma \Rightarrow \Delta) \mid \mmod{\phi} = [0]_n}\) (the \(\Sigma^s_i\)s) and a possible choice of a formula in \(\Union_{i \leq n} \nec^i \set{\phi \in \sub(\Gamma \Rightarrow \Delta) \mid \mmod{\phi} = [0]_n}\) (the \(\phi\)), we have that the possible number of premises of \((\modal[+]{\wKT{n}})\) in \(T\) is at most \(2^{mn}((n+1)m+1)\), i.e., finite.
  However, in \((w_i)_{i \in \mathbb{N}}\) there are infinitely many premises of \((\modal[+]{\wKT{n}})\), so at least there is a repeated  \(0\)-sequent, a contradiction.

  The created tree \(T\) is clearly a cyclic preproof.
  Given a node \(w\) of \(T\) let \(S_w\) be the sequent at \(w\) and \(R_w\) the rule at \(w\).
  Asumme \(w \in \rep(T)\) such that for each \(v \in (w^\circ,w]\) is non-progressing.
  Then for each \(v \in [w^\circ,w)\), \(R_v \neq (\modal[+]{\wKT{n}})\) so \(|S_{w^\circ}| > |S_w| = |S_{w^\circ}|\), a contradiction.
\end{proof}

We fix vocabularies \(V_+\) and \(V_-\) for which we want to calculate the uniform Lyndon interpolant.
The construction of the interpolant is divided in two parts.
First, we calculate a formula at each node \(w\) of the interpolation template, called the \emph{preinterpolant}.
The preinterpolants will still have some variables outside \(V_+\) and \(V_-\) that appear at repeat nodes.
However, from the preinterpolants we can obtain a positive modalized Kurahashi-Lyndon equational system of depth \(0\).
By applying the solution of this equational system to the preinterpolant at the root, we obtain the desired interpolant.
We finish this subsection by defining the preinterpolants.

\begin{definition}[Preinterpolant]
  Let \(T\) be an interpolation template and \(w\) a node of \(T\), let us denote the sequent at \(w\) as \(\Gamma_w \Rightarrow_{[i_w]_n} \Delta_w\).
  We will annotate each of the sequents of \(T\) with a formula \(\kappa\) called the \emph{preinterpolant at \(\kappa\)}, denoted as \(\kappa : \Gamma_w \Rightarrow_{[i_w]_n} \Delta_w\).
  We proceed by induction on the tree structure of \(T\).
  If \(w\) is a repeat node of shape
  \[
    \AxiomC{\(\)}
    \RightLabel{\(\rep\)}
    \UnaryInfC{\(\Gamma_w \Rightarrow_{[0]_n} \Delta_w\)}
    \DisplayProof
  \]
  then \(\kappa_w\) will be a fresh propositional variable \(x_w\).
  This variable shall not appear in any formula of \(T\) and for repeat nodes \(w,v\) such that \(w \neq v\) we must have that \(x_w \neq x_v\).
  In case \(w\) is not a repeat node we proceed by cases on the rule at \(w\), we describe the construction pictorically as follows (the preinterpolant at the conclusion is defined recursively from the preinterpolants at the premises):
  \[
    \AxiomC{\(\)}
    \RightLabel{\(\ax\)}
    \UnaryInfC{\(\bot : p, \Gamma \Rightarrow_{[i]_n} p, \Delta\)}
    \DisplayProof
    \qquad
    \AxiomC{}
    \RightLabel{\(\emp\)}
    \UnaryInfC{\(\top : {\Rightarrow_{[i]_n}} \)}
    \DisplayProof
    \qquad
    \AxiomC{\(\)}
    \RightLabel{\(\botL\)}
    \UnaryInfC{\(\bot : \bot, \Gamma \Rightarrow_{[i]_n} \Delta\)}
    \DisplayProof
    \qquad
    \AxiomC{\(\kappa : \Gamma \Rightarrow_{[i]_n} \bot, \Delta\)}
    \RightLabel{\(\botR\)}
    \UnaryInfC{\(\kappa : \Gamma \Rightarrow_{[i]_n} \bot, \Delta\)}
    \DisplayProof
  \]

  \[
    \AxiomC{\(\kappa_0 : \Gamma \Rightarrow_{[i]_n} \phi, \Delta\)}
    \AxiomC{\(\kappa_1 : \psi, \Gamma \Rightarrow_{[i]_n} \Delta\)}
    \RightLabel{\(\toL\)}
    \BinaryInfC{\(\kappa_0 \vee \kappa_1 : \phi \to \psi, \Gamma \Rightarrow_{[i]_n} \Delta\)}
    \DisplayProof
    \qquad
    \AxiomC{\(\kappa : \phi, \Gamma \Rightarrow_{[i]_n} \psi, \Delta\)}
    \RightLabel{\(\toL\)}
    \UnaryInfC{\(\kappa : \Gamma \Rightarrow_{[i]_n} \phi \to \psi, \Delta\)}
    \DisplayProof
  \]

  \[
    \AxiomC{\(\kappa^{\nec} : \necd^n \Sigma^s_0, \set{\nec^{j} \Sigma^s_j}_{1 \leq j < n} \Rightarrow_{[i]_n} \)}
    \AxiomC{\(\left[\kappa^{\nec}_\phi : \necd^n \Sigma^s_0, \set{\nec^{j} \Sigma^s_j}_{1 \leq j < n} \Rightarrow_{[i]_n} \phi\right]_{\phi \in \Theta}\)}
    \RightLabel{\(\modal[+]{\wKT{n}}\)}
    \BinaryInfC{\(\kappa : \set{\nec^{j+1} \Sigma_j}_{j < n}, \Gamma \Rightarrow_{[i+1]_n} \nec \Theta, \Delta\)}
    \DisplayProof
  \]
  where \(\kappa := \nec \kappa^{\nec} \wedge \bigwedge_{\phi \in \Theta} \pos \kappa^{\nec}_\phi \wedge \bigwedge (\Gamma \inter V_+) \wedge \bigwedge \neg (\Delta \inter V_-)\).
\end{definition}

\subsection{Interpolant}

Once we have the preinterpolants defined all left to do is to solve the equational system of the template to get the interpolant.
For that purpose let \(T\) be an interpolation template, for each node \(w\) let \(\kappa_w\) be the preinterpolant at \(w\).
Recall that the \emph{height of a node \(w\)}, denoted \(\hg(w)\), is the height of the subtree generated at \(w\).
Then, the leaves have height \(0\) and the height of the root is the same as the height of the tree.
For each \(i \in \mathbb{N}\) let \(\bar{x}_i\) be an arbitrary enumeration of \(\set{x_w \mid w \in \rep(T), \hg(w^\circ) = i}\).
If \(\hg(T) = H\), then we define \(\bar{x}_T = \bar{x}_0 \bar{x}_1 \cdots \bar{x}_H\).
Note that \(\bar{x}_T\) is an enumeration of \(\set{x_w \mid w \in \rep(T)}\).
We define the \emph{equational system of \(T\)}, denoted \(\mathcal{E}_T\), as
\[
  \set{(x_w, +, \kappa_{w^\circ}) \mid w \in \rep(T)}.
\]
In the following lemma, we show that \(\mathcal{E}_T\) is sovable in \(\wGL{n}\).

\begin{lemma}
  \(\mathcal{E}_T\) is a positive modalized Kurahashi-Lyndon equation system  of depth \(0\) over\\ \((\bar{x}_T, V_+, V_-, x_w \mapsto \set{0,\ldots,n-1})\).
  As a corollary it has a solution in \(\wGL{n}\).
\end{lemma}
\begin{proof}
  For each node \(w\) of \(T\) let \(\kappa_w\) be the preinterpolant at \(w\).
  By induction on the height of \(w\) we can show that \(\voc_+(\kappa_w) \subseteq V_+ \union \set{x_v \mid v \in \rep(T)}\), \(\voc_-(\kappa_w) \subseteq V_-\).
  Additionally, it is trivially true that \(\dep_n(x_w) \subseteq \set{0,\ldots,n-1}\) for \(w \in \rep(T)\) (note that we do not care about the depth of the solution for uniform Lyndon interpolation).
  Then, it is clear that \(\mathcal{E}_T\) is a Kurahashi--Lyndon equational system over \((\bar{x}_T, V_+,V_-, x_w \mapsto \mathbb{Z}/\mathbb{Z}n)\). Let us show that it is modalized positive and of depth \(0\).
  It is straightforward that it is positive.

  Proof that \(\mathcal{E}_T\) is of depth \(0\).
  For a node \(w\) of \(T\) let \(\Gamma_w \Rightarrow_{[i_w]_n} \Delta_w\) be the sequent at \(w\) in \(T\).
  By induction on \(w\), we can show that for any \(v \in \rep(T)\) we have that \(\dep_n(\kappa_w, x_v) \subseteq \set{[i_w]_n}\).
  As for any \(w \in \rep(T)\), \(w^\circ\) must be a \(0\)-sequent, we have that \(\dep_n(\kappa_{w^\circ}, x_v) \subseteq \set{[0]_n}\).

  Proof that \(\mathcal{E}_T\) is modalized.
  It suffices to prove that for every \(w,v \in \rep(T)\) if \(\hg(v^\circ) \leq \hg(w^\circ)\) then \(\kappa_{w^\circ}\) is modalized in \(x_v\).
  We have the following facts, which are easy to establish:
  \begin{enumerate}
    \item for each \(v \in \rep(T)\) and node \(w\) if \(x_v\) occurs at \(\kappa_w\) then \(w \leq v\),
    \item for each node \(w\) with rule \((\modal[+]{\wGL{n}})\) we have that \(\kappa_v\) is modalzied in \(\set{x_w \mid w \in \rep(T)}\),
    \item for each \(v \in \rep(T)\) and node \(w \not \in \rep(T)\) with immediate successors \(w_0,\ldots,w_{n-1}\), if \(\kappa_{w_0}\), \ldots, \(\kappa_{w_{n-1}}\) are modalized in \(x_v\) then \(\kappa_w\) is modalized in \(x_v\).
  \end{enumerate}
  If \(\hg(v^\circ) \leq \hg(w^\circ)\), we have that either \(v^\circ \geq w^\circ\) or \(v^\circ\) and \(w^\circ\) are incomparable.
  Assume \(v^\circ\) and \(w^\circ\) are incomparable, then it must also be the case that \(v\) and \(w^\circ\) are incomparable; so \(\kappa_{w^\circ}\) is modalized in \(x_v\) as \(x_v\) does not occur in \(\kappa_{w^\circ}\).
  Finally, assume that \(v^\circ \geq w^\circ\). We know there must be a node \(u \in [v^\circ,v)\) which is annotated with \((\modal[+]{\wGL{n}})\); so \(\kappa_u\) is modalized in \(x_v\).
  Then, using the third fact, we can show  that (by induction) the preinterpolant of any node below \(u\) is modalized in \(x_v\) (since it is not a repeat node and all the preinterpolans of the immediate successors are modalized in \(x_v\), either by virtue of the node not being comprable with \(v\) or by the induction hypothesis).
  We can conclude, as \(w^\circ \leq v^\circ \leq u\), that \(\kappa_{w^\circ}\) is modalized in \(x_v\).
\end{proof}

Finally, we have all the tools to define the interpolant.

\begin{definition}[Interpolant]
  Given an interpolation template \(T\), we define the interpolant of \(T\), denoted by~\(\iota_T\), as the formula obtained by applying the substitution solving \(\mathcal{E}_T\) to the preinterpolant at the root of~$T$.
\end{definition}

Strictly speakening the definition of the interpolant is not unique.
For an interpolation template \(T\),  its interpolant~\(\iota_T\) will depend on the choice of the solution of \(T\).
However, any two interpolants will be \(\wGL{n}\)-logically equivalent, so this lack of unicity should not worry us.

Once the interpolant is defined, we need to verify that it fulfills the necessary interpolant properties.
By the definition it is straightforward that \(\voc_+(\iota_T) \subseteq V_+\) and \(\voc_-(\iota_T) \subseteq V_-\), as \(\iota_T\) is obtained by applying the solution of an equational system over \((\bar{x}_T, V_+,V_-, x_w \mapsto \set{0,\ldots,n-1})\) to a formula \(\kappa\) with \(\voc_+(\kappa) \subseteq V_+ \union \set{x_w \mid w \in \rep(T)}\) and \(\voc_-(\kappa) \subseteq V_-\) (using Lemma~\ref{substitution-and-polarity}).
The other two properties of the interpolants will be obtained thanks to the following two theorems.

\begin{lemma}\label{first-verification-interpolant-wGL}
  Let \(T\) be a template for \(\Gamma \Rightarrow \Delta\), then \(\n{\wGL{n}} \vdash \Gamma \Rightarrow \Delta, \iota_T\).
\end{lemma}
\begin{proof}
  Given a node \(w\) of \(T\) we will write \(\Gamma_w \Rightarrow_{[i_w]_n} \Delta_w\) to mean the sequent at \(w\) and \(\kappa_w\) the preinterpolant at \(w\).
  Let \((\cdot)^*\) be a solution of \(\mathcal{E}_T\).
  We will define a function \(\alpha\) such that given a node \(w\) of \(T\), \(\alpha(w)\) is a proof of \(\Gamma_w \Rightarrow \Delta_w, \kappa^*_w\) in \(\n{\wGL{n}} + \cut + \wk\).
  Then we can use the eliminability of \((\wk)\) in \(\n{\wGL{n}} + \cut\) and the eliminability of \((\cut)\) in \(\n{\wGL{n}}\) to obtain the desired proof.

  We will define \(\alpha\) via corecursion.
  Once \(\alpha\) is defined it will be clear by construction that \(\alpha(w)\) is a preproof.
  We will then argue that it is, in addition, a proof.
  
  First, assume that \(w \in \rep(T)\).
  The function is defined as
  \[
    \AxiomC{\(\)}
    \RightLabel{\(\rep\)}
    \UnaryInfC{\(x_w : \Gamma_w \Rightarrow_{[0]_n} \Delta_w\)}
    \DisplayProof
    \quad
    \overset{\alpha}{\longmapsto}
    \quad
    \AxiomC{\(\alpha(w^\circ)\)}
    \noLine
    \UnaryInfC{\(\Gamma_w \Rightarrow \Delta_w, \kappa_{w^\circ}^*\)}
    \RightLabel{\(\wk\)}
    \UnaryInfC{\(\Gamma_w \Rightarrow \Delta_w, x^*_w, \kappa_{w^\circ}^*\)}
    \AxiomC{\(\tau\)}
    \noLine
    \UnaryInfC{\(\kappa_{w^\circ}^*, \Gamma_w \Rightarrow \Delta_w, x^*_w\)}
    \RightLabel{\(\cut\)}
    \BinaryInfC{\(\Gamma_w \Rightarrow \Delta_w, x^*_w\)}
    \DisplayProof
  \]
  where \(\tau\) is a proof of \(\kappa_{w^\circ}^*, \Gamma_w \Rightarrow \Delta_w, x^*_w\) in \(\n{\wGL{n}}\) which exists since \(\wGL{n} \vdash x^*_w \leftrightarrow \kappa^*_{w^\circ}\).

  Now let \(R\) be the rule at \(w\).
  Below you can find the definiton if \(R\) is one of \((\ax)\), \((\emp)\), \((\botL)\), \((\botR)\), \((\toL)\), \((\toR)\).
  \[
    \AxiomC{\(\)}
    \RightLabel{\(\ax\)}
    \UnaryInfC{\(\bot : p, \Gamma \Rightarrow_{[i]_n} p, \Delta\)}
    \DisplayProof
    \quad
    \overset{\alpha}{\longmapsto}
    \quad
    \AxiomC{}
    \RightLabel{\(\ax\)}
    \UnaryInfC{\(p, \Gamma \Rightarrow p, \Delta, \bot\)}
    \DisplayProof
  \]

  \[
    \AxiomC{\(\)}
    \RightLabel{\(\emp\)}
    \UnaryInfC{\(\top :  {\Rightarrow_{[i]_n} }\)}
    \DisplayProof
    \quad
    \overset{\alpha}{\longmapsto}
    \quad
    \AxiomC{}
    \RightLabel{\(\botL\)}
    \UnaryInfC{\(\bot \Rightarrow \bot\)}
    \RightLabel{\(\toR\)}
    \UnaryInfC{\( \Rightarrow \top\)}
    \DisplayProof
  \]

  \[
    \AxiomC{\(\)}
    \RightLabel{\(\botL\)}
    \UnaryInfC{\(\bot : \bot, \Gamma \Rightarrow_{[i]_n} \Delta\)}
    \DisplayProof
    \quad
    \overset{\alpha}{\longmapsto}
    \quad
    \AxiomC{}
    \RightLabel{\(\botL\)}
    \UnaryInfC{\(\bot, \Gamma \Rightarrow \Delta, \bot\)}
    \DisplayProof
  \]

  \[
    \AxiomC{\(w_0\)}
    \noLine
    \UnaryInfC{\(\kappa :  \Gamma \Rightarrow_{[i]_n} \Delta\)}
    \RightLabel{\(\botR\)}
    \UnaryInfC{\(\kappa : \Gamma \Rightarrow_{[i]_n} \Delta, \bot\)}
    \DisplayProof
    \quad
    \overset{\alpha}{\longmapsto}
    \quad
    \AxiomC{\(\alpha(w_0)\)}
    \noLine
    \UnaryInfC{\(\Gamma \Rightarrow \Delta, \kappa^*\)}
    \RightLabel{\(\botR\)}
    \UnaryInfC{\(\Gamma \Rightarrow \Delta, \bot, \kappa^*\)}
    \DisplayProof
  \]

  \[
    \AxiomC{\(w_0\)}
    \noLine
    \UnaryInfC{\(\kappa_0 :  \Gamma \Rightarrow_{[i]_n} \phi, \Delta\)}
    \AxiomC{\(w_0\)}
    \noLine
    \UnaryInfC{\(\kappa_1 :  \psi, \Gamma \Rightarrow_{[i]_n}  \Delta\)}
    \RightLabel{\(\toL\)}
    \BinaryInfC{\(\kappa_0 \vee \kappa_1 : \phi \to \psi, \Gamma \Rightarrow_{[i]_n} \Delta\)}
    \DisplayProof
    \quad
    \overset{\alpha}{\longmapsto}
    \quad
    \AxiomC{\(\alpha(w_0)\)}
    \noLine
    \UnaryInfC{\(\Gamma \Rightarrow \phi, \Delta, \kappa^*_0\)}
    \RightLabel{\(\wk\)}
    \UnaryInfC{\(\Gamma \Rightarrow \phi, \Delta, \kappa^*_0, \kappa^*_1\)}
    \RightLabel{\(\veeR\)}
    \UnaryInfC{\(\Gamma \Rightarrow \phi, \Delta, \kappa^*_0 \vee \kappa^*_1\)}
    \AxiomC{\(\alpha(w_1)\)}
    \noLine
    \UnaryInfC{\(\psi,\Gamma \Rightarrow  \Delta, \kappa^*_1\)}
    \RightLabel{\(\wk\)}
    \UnaryInfC{\(\psi, \Gamma \Rightarrow \Delta, \kappa^*_0, \kappa^*_1\)}
    \RightLabel{\(\veeR\)}
    \UnaryInfC{\(\psi, \Gamma \Rightarrow \Delta, \kappa^*_0 \vee \kappa^*_1\)}
    \RightLabel{\(\toL\)}
    \BinaryInfC{\(\phi \to \psi, \Gamma \Rightarrow \Delta,\kappa_0^* \vee \kappa_1^*\)}
    \DisplayProof
  \]

  \[
    \AxiomC{\(w_0\)}
    \noLine
    \UnaryInfC{\(\kappa :  \phi, \Gamma \Rightarrow_{[i]_n} \psi, \Delta\)}
    \RightLabel{\(\toR\)}
    \UnaryInfC{\(\kappa : \Gamma \Rightarrow_{[i]_n} \phi \to \psi, \Delta\)}
    \DisplayProof
    \quad
    \overset{\alpha}{\longmapsto}
    \quad
    \AxiomC{\(\alpha(w_0)\)}
    \noLine
    \UnaryInfC{\(\phi, \Gamma \Rightarrow \psi, \Delta, \kappa^*\)}
    \RightLabel{\(\toR\)}
    \UnaryInfC{\(\Gamma \Rightarrow \phi \to \psi, \Delta, \kappa^*\)}
    \DisplayProof
  \]

  Finally assume that \(R = (\modal[+]{\wKT{n}})\).
  In this case \(w\) is of shape
  \[
    \AxiomC{\(\begin{matrix} w^{\nec} \\ \kappa^{\nec} : \necd^n \Sigma_0, \set{\nec^{i} \Sigma_i}_{1 \leq i < n} \Rightarrow_{[i]_n} \end{matrix}\)}
    \AxiomC{\(\left[\begin{matrix} w^{\nec}_\phi \\ \kappa^{\nec}_\phi : \necd^n \Sigma_0, \set{\nec^{i} \Sigma_i}_{1 \leq i < n} \Rightarrow_{[i]_n} \phi\end{matrix}\right]_{\phi \in \Theta}\)}
    \RightLabel{\(\modal[+]{\wKT{n}}\)}
    \BinaryInfC{\(\nec \kappa^{\nec} \wedge \bigwedge_{\phi \in \Theta} \pos \kappa^{\nec}_\phi \wedge \bigwedge (\Gamma \inter V_+) \wedge \bigwedge \neg (\Delta \inter V_-) : \set{\nec^{i+1} \Sigma_i}_{i < n}, \Gamma \Rightarrow_{[i+1]_n} \nec \Theta, \Delta\)}
    \DisplayProof
  \]
  Since \(\kappa^*_w\) is a conjunction of formulas, it suffices to construct a preproof for each conjunct and join them with multiple applications of \((\wedgeR)\).
  \begin{itemize}
    \item Preproof for \(\nec \kappa^{\nec}\).
      The desired preproof is
      \[
        \AxiomC{\(\alpha(w^{\nec})\)}
        \noLine
        \UnaryInfC{\(\necd^n \Sigma_0, \set{\nec^{i} \Sigma_i}_{1 \leq i < n} \Rightarrow (\kappa^{\nec})^*\)}
        \RightLabel{\(\wk\)}
        \UnaryInfC{\(\set{\necd^n \nec^{i} \Sigma_i}_{i < n} \Rightarrow (\kappa^{\nec})^*\)}
        \RightLabel{\(\modal[n]{\KT}\)}
        \UnaryInfC{\(\set{\nec^{i+1} \Sigma_i}_{i < n}, \Gamma \Rightarrow \nec \Theta, \Delta, \nec (\kappa^{\nec})^*\)}
        \DisplayProof
      \]
    \item Preproof for \(\pos \kappa^{\nec}_{\phi}\) for \(\phi \in \Theta\).
      The desired preproof is
      \[
        \AxiomC{\(\alpha(w^{\nec}_\phi)\)}
        \noLine
        \UnaryInfC{\(\necd^n \Sigma_0, \set{ \nec^{i} \Sigma_i}_{1 \leq i < n}\Rightarrow \phi, (\kappa^{\nec}_\phi)^*\)}
        \RightLabel{\(\negL\)}
        \UnaryInfC{\(\neg (\kappa^{\nec}_\phi)^*, \necd^n \Sigma_0, \set{ \nec^{i} \Sigma_i}_{1 \leq i < n}\Rightarrow \phi\)}
        \RightLabel{\(\wk\)}
        \UnaryInfC{\(\necd^n \neg (\kappa^{\nec}_\phi)^*,\set{\necd^n \nec^{i} \Sigma_i}_{i < n}\Rightarrow \phi\)}
        \RightLabel{\(\modal[n]{\wKT{n}}\)}
        \UnaryInfC{\(\nec \neg (\kappa^{\nec}_\phi)^*,\set{\nec^{i+1} \Sigma_i}_{i < n}, \Gamma \Rightarrow \nec \Theta, \Delta\)}
        \RightLabel{\(\negR\)}
        \UnaryInfC{\(\set{\nec^{i+1} \Sigma_i}_{i < n}, \Gamma \Rightarrow \nec \Theta, \Delta, \pos (\kappa^{\nec}_\phi)^*\)}
        \DisplayProof
      \]
    \item Preproof for \(p\) where \(p \in \Gamma\).
      We note that \(p^* = p\), so the desired preproof is
      \[
        \AxiomC{}
        \RightLabel{\(\ax\)}
        \UnaryInfC{\(\set{\nec^{i+1} \Sigma_i}_{i < n}, \Gamma \Rightarrow \nec \Theta, \Delta, p\)}
        \DisplayProof
      \]
    \item Preproof for \(\neg p\) where \(p \in \Delta\).
      We note that \(p^* = p\), so the desired preproof is
      \[
        \AxiomC{}
        \RightLabel{\(\ax\)}
        \UnaryInfC{\(p, \set{\nec^{i+1} \Sigma_i}_{i < n}, \Gamma \Rightarrow \nec \Theta, \Delta\)}
        \RightLabel{\(\negR\)}
        \UnaryInfC{\( \set{\nec^{i+1} \Sigma_i}_{i < n}, \Gamma \Rightarrow \nec \Theta, \Delta, \neg p\)}
        \DisplayProof
      \]
  \end{itemize}

  To show that \(\alpha(w)\) is always a proof, we assign a measure \(\omega |\Gamma_w \Rightarrow_{[i_w]_n} \Delta_w| + \length(w)\) to every node \(w\) of~\(T\), where \(|\Gamma_w \Rightarrow_{[i_w]_n} \Delta_w|\) is the sum of the complexity of the formulas in \(\Gamma_w \Rightarrow_{[i_w]_n} \Delta_w\) and \(\length(w)\) is the length of \(w\) as a sequence.
  We notice that in the corecursive calls of \(\alpha\), this measure strictly decrease in all cases except when \(w\) has a rule \((\modal[+]{\wKT{n}})\).
  However, in the produced preproof in this case,  there is an application of \((\modal{\wKT{n}})\) in the path from the corecursive call to the root. 
  This implies that in any infinite branch of \(\alpha(w)\), there will be infinitely many applications of \((\modal{\wKT{n}})\), as desired.
\end{proof}

\begin{lemma}\label{second-verification-interpolant-wGL}
  Let \(T\) be a template for \(\Gamma \Rightarrow \Delta\).
  For any \(\Xi \Rightarrow \Lambda\) such that \(\voc_+(\Xi \Rightarrow \Lambda) \subseteq V_+\) and \(\voc_-(\Xi \Rightarrow \Lambda) \subseteq V_-\), we have that \(\nmod{\wGL{n}} \vdash \Gamma, \Xi \Rightarrow \Delta, \Lambda\) implies \(\n{\wGL{n}} \vdash \iota_T, \Xi \Rightarrow \Lambda\).
\end{lemma}
\begin{proof}
  Given a node \(w\) of \(T\), we write \(\Gamma_w \Rightarrow_{[i_w]_n} \Delta_w\) to denote the sequent at \(w\) and \(\kappa_w\) for the preinterpolant at \(w\).
  Let \((\cdot)^*\) be a solution of \(\mathcal{E}_T\).
  We will define a function \(\beta\) such that, given a node \(w\) of \(T\) and a proof \(\pi \vdash \Gamma_w, \Xi \Rightarrow \Delta_w, \Lambda\) in \(\nmod{\wGL{n}}\) with \(\voc_b(\Xi \Rightarrow \Lambda) \subseteq V_b\) for \(b \in \set{+,-}\),  \(\beta(w, \pi)\) is a proof of \(\kappa^*_w, \Xi \Rightarrow \Lambda\) in \(\n{\wGL{n}} + \cut + \wk\).
  Then we can use the eliminability of \((\wk)\) in \(\n{\wGL{n}} + \cut\) and the eliminability of \((\cut)\) in \(\n{\wGL{n}}\) to obtain the desired proof.

  We will first define \(\beta\) via corecursion. It will be clear by the construction that \(\beta(w,\pi)\) is a preproof.
  Then we will argue that it is a proof.
  
  First, assume that \(w \in \rep(T)\).
  Then the function \(\beta\) is defined as
  \[
    \left(
    \AxiomC{}
    \RightLabel{\(\rep\)}
    \UnaryInfC{\(x_w : \Gamma_w \Rightarrow_{[0]_n} \Delta_w\)}
    \DisplayProof,
    \quad
    \AxiomC{\(\pi\)}
    \noLine
    \UnaryInfC{\(\Gamma_w, \Xi \Rightarrow \Delta_w, \Lambda\)}
    \DisplayProof
    \right)
    \quad \overset{\beta}{\longmapsto} \quad
    \AxiomC{\(\tau\)}
    \noLine
    \UnaryInfC{\( x^*_w, \Xi \Rightarrow \Lambda, \kappa^*_{w^\circ}\)}
    \AxiomC{\(\beta(w^\circ,\pi)\)}
    \noLine
    \UnaryInfC{\(\kappa^*_{w^\circ}, \Xi \Rightarrow \Lambda\)}
    \RightLabel{\(\wk\)}
    \UnaryInfC{\(\kappa^*_{w^\circ}, x^*_w, \Xi \Rightarrow \Lambda\)}
    \RightLabel{\(\cut\)}
    \BinaryInfC{\(x^*_w, \Xi \Rightarrow \Lambda\)}
    \DisplayProof
  \]
  where \(\tau\) is a proof of \( x^*_w, \Xi \Rightarrow \Lambda, \kappa^*_{w^\circ}\) in \(\n{\wGL{n}}\), which exists since \(\wGL{n} \vdash x^*_w \leftrightarrow \kappa^*_{w^\circ}\).

  Now let \(R\) be the rule at \(w\).
  Below there is the defintion if \(\beta(w,\pi)\) is \(R\) is one of \((\ax)\), \((\emp)\), \((\botL)\), \((\botR)\), \((\toL)\) or \((\toR)\).
  \[
    \left(
      \AxiomC{}
      \RightLabel{\(\ax\)}
      \UnaryInfC{\(\bot : p, \Gamma \Rightarrow_{[i]_n} p, \Delta\)}
      \DisplayProof,
      \quad
      \AxiomC{\(\pi\)}
      \noLine
      \UnaryInfC{\(p, \Gamma, \Xi \Rightarrow p, \Delta, \Lambda\)}
      \DisplayProof
    \right)
    \quad
    \overset{\beta}{\longmapsto}
    \quad
    \AxiomC{\(\)}
    \RightLabel{\(\botL\)}
    \UnaryInfC{\(\bot, \Xi \Rightarrow \Delta\)}
    \DisplayProof
  \] 

  \[
    \left(
      \AxiomC{}
      \RightLabel{\(\emp\)}
      \UnaryInfC{\(\top : {\Rightarrow_{[i]_n}}\)}
      \DisplayProof,
      \quad
      \AxiomC{\(\pi\)}
      \noLine
      \UnaryInfC{\(\Xi \Rightarrow \Lambda\)}
      \DisplayProof
    \right)
    \quad
    \overset{\beta}{\longmapsto}
    \quad
    \AxiomC{\(\pi\)}
    \noLine
    \UnaryInfC{\(\Xi \Rightarrow \Lambda\)}
    \RightLabel{\(\wk\)}
    \UnaryInfC{\(\top, \Xi \Rightarrow \Delta\)}
    \DisplayProof
  \]

  \[
    \left(
      \AxiomC{}
      \RightLabel{\(\botL\)}
      \UnaryInfC{\(\bot : \bot, \Gamma \Rightarrow_{[i]_n} \Delta\)}
      \DisplayProof,
      \quad
      \AxiomC{\(\pi\)}
      \noLine
      \UnaryInfC{\(\bot, \Gamma, \Xi \Rightarrow \Delta, \Lambda\)}
      \DisplayProof
    \right)
    \quad
    \overset{\beta}{\longmapsto}
    \quad
    \AxiomC{\(\)}
    \RightLabel{\(\botL\)}
    \UnaryInfC{\(\bot, \Xi \Rightarrow \Delta\)}
    \DisplayProof
  \]

  \[
    \left(
      \AxiomC{\(w_0\)}
      \noLine
      \UnaryInfC{\(\kappa : \Gamma \Rightarrow_{[i]_n} \Delta\)}
      \RightLabel{\(\botR\)}
      \UnaryInfC{\(\kappa : \Gamma \Rightarrow_{[i]_n} \bot, \Delta\)}
      \DisplayProof,
      \quad
      \AxiomC{\(\pi\)}
      \noLine
      \UnaryInfC{\( \Gamma, \Xi \Rightarrow \bot,\Delta, \Lambda\)}
      \DisplayProof
    \right)
    \quad
    \overset{\beta}{\longmapsto}
    \quad
    \AxiomC{\(\beta(w_0, \inv_{\botR}(\pi))\)}
    \noLine
    \UnaryInfC{\(\kappa^*, \Xi \Rightarrow \Delta\)}
    \RightLabel{\(\wk\)}
    \UnaryInfC{\(\kappa^*, \Xi \Rightarrow \Delta\)}
    \DisplayProof
  \]

  \begin{multline*}
    \left(
      \AxiomC{\(w_0\)}
      \noLine
      \UnaryInfC{\(\kappa_0 : \Gamma \Rightarrow_{[i]_n} \phi, \Delta\)}
      \AxiomC{\(w_0\)}
      \noLine
      \UnaryInfC{\(\kappa_1 : \psi, \Gamma \Rightarrow_{[i]_n} \Delta\)}
      \RightLabel{\(\toL\)}
      \BinaryInfC{\(\kappa_0 \vee \kappa_1 :  \phi \to \psi, \Gamma \Rightarrow_{[i]_n} \Delta\)}
      \DisplayProof,
      \quad
      \AxiomC{\(\pi\)}
      \noLine
      \UnaryInfC{\( \phi \to \psi,\Gamma, \Xi \Rightarrow \Delta, \Lambda\)}
      \DisplayProof
    \right)
    \quad
    \overset{\beta}{\longmapsto}\\
    \AxiomC{\(\beta(w_0, \inv_{\toL_0}(\pi))\)}
    \noLine
    \UnaryInfC{\(\kappa^*_0, \Xi \Rightarrow \Delta\)}
    \AxiomC{\(\beta(w_1, \inv_{\toL_1}(\pi))\)}
    \noLine
    \UnaryInfC{\(\kappa^*_1, \Xi \Rightarrow \Delta\)}
    \RightLabel{\(\veeL\)}
    \BinaryInfC{\(\kappa^*_0 \vee \kappa^*_1, \Xi \Rightarrow \Delta\)}
    \DisplayProof
  \end{multline*}

  \[
    \left(
      \AxiomC{\(w_0\)}
      \noLine
      \UnaryInfC{\(\kappa : \phi, \Gamma \Rightarrow_{[i]_n} \psi, \Delta\)}
      \RightLabel{\(\toR\)}
      \UnaryInfC{\(\kappa : \Gamma \Rightarrow_{[i]_n} \phi \to \psi, \Delta\)}
      \DisplayProof,
      \quad
      \AxiomC{\(\pi\)}
      \noLine
      \UnaryInfC{\( \Gamma, \Xi \Rightarrow \phi \to \psi,\Delta, \Lambda\)}
      \DisplayProof
    \right)
    \quad
    \overset{\beta}{\longmapsto}
    \quad
    \AxiomC{\(\beta(w_0, \inv_{\toR}(\pi))\)}
    \noLine
    \UnaryInfC{\(\kappa^*, \Xi \Rightarrow \Delta\)}
    \RightLabel{\(\wk\)}
    \UnaryInfC{\(\kappa^*, \Xi \Rightarrow \Delta\)}
    \DisplayProof
  \]

  Finally, assume that \(R = (\modal[+]{\wGL{n}})\).
  In this case \(w\) is of shape
  \[
    \AxiomC{\(\begin{matrix} w^{\nec} \\ \kappa^{\nec} : \necd^n \Sigma_0, \set{\nec^{i} \Sigma_i}_{1 \leq i < n} \Rightarrow_{[i]_n} \end{matrix}\)}
    \AxiomC{\(\left[\begin{matrix} w^{\nec}_\phi \\ \kappa^{\nec}_\phi : \necd^n \Sigma_0, \set{\nec^{i} \Sigma_i}_{1 \leq i < n} \Rightarrow_{[i]_n} \phi\end{matrix}\right]_{\phi \in \Theta}\)}
    \RightLabel{\(\modal[+]{\wKT{n}}\)}
    \BinaryInfC{\(\nec \kappa^{\nec} \wedge \bigwedge_{\phi \in \Theta} \pos \kappa^{\nec}_\phi \wedge \bigwedge (\Gamma \inter V_+) \wedge \bigwedge \neg (\Delta \inter V_-) : \set{\nec^{i+1} \Sigma_i}_{i < n}, \Gamma \Rightarrow_{[i+1]_n} \nec \Theta, \Delta\)}
    \DisplayProof
  \]
  where \(\Gamma_w = \set{\nec^{i+1} \Sigma_i}_{i < n}, \Gamma\), \(\Delta_w = \nec \Theta, \Delta\) and \(\Gamma,\Delta \subseteq \text{Var}\), \(\Gamma \inter \Delta = \varnothing\).
  We proceed by a further subcase analysis on the last rule applied to \(\pi\).
  \begin{itemize}
    \item Last rule of \(\pi\) is \((\ax)\).
      Then there is a variable \(p\) occurring to the left side and to the right side of the conclusion of \(\pi\).
      We make cases in the possible multisets \(p\) may belong to.
      \begin{itemize}
        \item Case \(p \in \Gamma \cap \Delta\). This case is impossible since \(\Gamma \cap \Delta = \varnothing\).
        \item Case \(p \in \Gamma \cap \Lambda\).
          Since \(p \in \Lambda\) and \(\voc_+(\Xi \Rightarrow \Lambda) \subseteq V_+\) we have that \(p \in \Gamma \cap V_+\), so \(p\) is a conjunct of \(\kappa^*_w\).
          Then \(\beta(w,\pi)\) is defined as
          \[
            \AxiomC{}
            \RightLabel{\(\ax\)}
            \UnaryInfC{\(p, \Xi \Rightarrow \Lambda\)}
            \RightLabel{\(\wk + \wedgeL\)}
            \doubleLine
            \UnaryInfC{\(\kappa^*_w, \Xi \Rightarrow \Lambda\)}
            \DisplayProof
          \]
        \item Case \(p \in \Xi \cap \Delta\).
          Since \(p \in \Xi\) and \(\voc_-(\Xi \Rightarrow \Lambda) \subseteq V_-\) we have that \(p \in \Delta \cap V_-\), so \(\neg p\) is a conjunct of \(\kappa^*_w\).
          Then \(\beta(w,\pi)\) is defined as
          \[
            \AxiomC{}
            \RightLabel{\(\ax\)}
            \UnaryInfC{\(\Xi \Rightarrow p, \Lambda\)}
            \RightLabel{\(\negL\)}
            \UnaryInfC{\( \neg p, \Xi \Rightarrow\Lambda\)}
            \RightLabel{\(\wk + \wedgeL\)}
            \doubleLine
            \UnaryInfC{\(\kappa^*_w, \Xi \Rightarrow \Lambda\)}
            \DisplayProof
          \]
        \item Case \(p \in \Xi \cap \Lambda\).
          In this case \(\beta(w,\pi)\) is defined as
          \[
            \AxiomC{}
            \RightLabel{\(\ax\)}
            \UnaryInfC{\(\kappa^*_w, \Xi \Rightarrow \Lambda\)}
            \DisplayProof
          \]
      \end{itemize}
    \item Last rule of \(\pi\) is \(\botL\).
      Since \(\bot \not \in \Gamma\) it must be the case that \(\bot \in \Xi\).
      Then \(\beta(w,\pi)\) is defined as
      \[
        \AxiomC{\(\)}
        \RightLabel{\(\botL\)}
        \UnaryInfC{\(\kappa^*_w, \Xi \Rightarrow \Lambda\)}
        \DisplayProof
      \]
    \item Last rule of \(\pi\) is \((\botR)\).
      As \(\Delta \subseteq \var\) it must be the case that \(\bot \in \Lambda\).
      Then \(\beta\) is defined as
      \[
        \left(
          w,
          \AxiomC{\(\pi_0\)}
          \noLine
          \UnaryInfC{\(\set{\nec^{i+1} \Sigma_i}_{i < n}, \Gamma, \Xi \Rightarrow \nec \Theta, \Delta, \Lambda'\)}
          \RightLabel{\(\botR\)}
          \UnaryInfC{\(\set{\nec^{i+1} \Sigma_i}_{i < n}, \Gamma, \Xi \Rightarrow \nec \Theta, \Delta, \bot, \Lambda'\)}
          \DisplayProof
        \right)
        \quad
        \overset{\beta}{\longmapsto}
        \quad
        \AxiomC{\(\beta(w,\pi_0)\)}
        \noLine
        \UnaryInfC{\(\kappa^*_w, \Xi \Rightarrow \Lambda'\)}
        \RightLabel{\(\botR\)}
        \UnaryInfC{\(\kappa^*_w, \Xi \Rightarrow \bot, \Lambda'\)}
        \DisplayProof
      \]
      where \(\Lambda = \bot, \Lambda'\).

    \item Last rule of \(\pi\) is \((\toL)\).
      As \(\Gamma \subseteq \var\) it must be the case that the principal formula of \((\toL)\) is in \(\Lambda\).
      Then \(\beta\) is defined as
      \begin{multline*}
        \left(
          w,
          \AxiomC{\(\pi_0\)}
          \noLine
          \UnaryInfC{\(\set{\nec^{i+1} \Sigma_i}_{i < n}, \Gamma, \Xi' \Rightarrow \nec \Theta, \Delta, \phi, \Lambda\)}
          \AxiomC{\(\pi_1\)}
          \noLine
          \UnaryInfC{\(\set{\nec^{i+1} \Sigma_i}_{i < n}, \Gamma, \psi, \Xi' \Rightarrow \nec \Theta, \Delta, \Lambda\)}
          \RightLabel{\(\toL\)}
          \BinaryInfC{\(\set{\nec^{i+1} \Sigma_i}_{i < n}, \Gamma, \phi \to \psi,\Xi' \Rightarrow \nec \Theta, \Delta,  \Lambda\)}
          \DisplayProof
        \right)
        \quad
        \overset{\beta}{\longmapsto}\\ \\
        \AxiomC{\(\beta(w,\pi_0)\)}
        \noLine
        \UnaryInfC{\(\kappa^*_w, \Xi' \Rightarrow \phi, \Lambda\)}
        \AxiomC{\(\beta(w,\pi_1)\)}
        \noLine
        \UnaryInfC{\(\kappa^*_w, \psi, \Xi' \Rightarrow \Lambda\)}
        \RightLabel{\(\toL\)}
        \BinaryInfC{\(\kappa^*_w, \phi \to \psi, \Xi' \Rightarrow  \Lambda\)}
        \DisplayProof
      \end{multline*}
      where \(\Xi = \phi \to \psi, \Xi'\).

    \item Last rule of \(\pi\) is \((\toR)\).
      As \(\Delta \subseteq \var\) it must be the case that the principal formula of \((\toR)\) is in \(\Lambda\).
      Then \(\beta\) is defined as
      \[
        \left(
          w,
          \AxiomC{\(\pi_0\)}
          \noLine
          \UnaryInfC{\(\set{\nec^{i+1} \Sigma_i}_{i < n}, \Gamma, \phi, \Xi \Rightarrow \nec \Theta, \Delta, \psi, \Lambda'\)}
          \RightLabel{\(\toR\)}
          \UnaryInfC{\(\set{\nec^{i+1} \Sigma_i}_{i < n}, \Gamma, \Xi \Rightarrow \nec \Theta, \Delta, \phi \to \psi, \Lambda'\)}
          \DisplayProof
        \right)
        \quad
        \overset{\beta}{\longmapsto}
        \quad
        \AxiomC{\(\beta(w,\pi_0)\)}
        \noLine
        \UnaryInfC{\(\kappa^*_w, \phi, \Xi \Rightarrow \psi, \Lambda'\)}
        \RightLabel{\(\toR\)}
        \UnaryInfC{\(\kappa^*_w, \Xi \Rightarrow \phi \to \psi, \Lambda'\)}
        \DisplayProof
      \]
      where \(\Lambda = \phi \to \psi, \Lambda'\).
    \item Last rule of \(\pi\) is \((\modalalt{\wKT{n}})\).
      First, assume that the principal formula of \((\modalalt{\wKT{n}})\) in \(\pi\) is in \(\nec \Theta\).
      Then we define \(\beta\) as
      \begin{multline*}
        \left(
          w,
          \AxiomC{\(\pi_0\)}
          \noLine
          \UnaryInfC{\(\necd^n \Sigma_0, \set{\nec^i \Sigma_i}_{1 \leq i < n}, \necd^n\Xi^{\nec}_0, \set{\nec^i \Xi^{\nec}_i}_{1 \leq i < n} \Rightarrow  \phi\)}
          \RightLabel{\(\modalalt{\wKT{n}}\)}
          \UnaryInfC{\(\set{\nec^{i+1} \Sigma_i}_{i < n}, \Gamma, \Xi \Rightarrow \nec \phi, \nec \Theta', \Delta, \Lambda\)}
          \DisplayProof
        \right)
        \overset{\beta}{\longmapsto}\\
        \AxiomC{\(\beta(w^{\nec}_\phi,\pi_0)\)}
        \noLine
        \UnaryInfC{\((\kappa^{\nec}_\phi)^*, \necd^n \Xi^{\nec}_0, \set{\nec^i \Xi^{\nec}_i}_{1 \leq i < n} \Rightarrow \)}
        \RightLabel{\(\negR\)}
        \UnaryInfC{\(\necd^n \Xi^{\nec}_0, \set{\nec^i \Xi^{\nec}_i}_{1 \leq i < n} \Rightarrow \neg (\kappa^{\nec}_\phi)^*\)}
        \RightLabel{\(\wk\)}
        \UnaryInfC{\(\necd^n \Xi^{\nec}_0, \set{\necd^n\nec^i \Xi^{\nec}_i}_{1 \leq i < n} \Rightarrow \neg (\kappa^{\nec}_\phi)^*\)}
        \RightLabel{\(\modal{\wKT{n}}\)}
        \UnaryInfC{\(\Xi \Rightarrow \Lambda, \nec \neg (\kappa^{\nec}_\phi)^*\)}
        \RightLabel{\(\negL\)}
        \UnaryInfC{\(\pos (\kappa^{\nec}_\phi)^*, \Xi \Rightarrow \Lambda \)}
        \RightLabel{\(\wk + \wedgeL\)}
        \doubleLine
        \UnaryInfC{\(\kappa^*_w, \Xi \Rightarrow \Lambda \)}
        \DisplayProof
      \end{multline*}
      where \(\mmod{\Sigma_i \union \Xi^{\nec}_i} \subseteq \set{[0]_n}\), \(\nec^{i+1} \Xi^{\nec}_{i} \subseteq \Xi\) and \(\Theta = \phi, \Theta'\).
      Now, assume that the principal formula of \((\modalalt{\wKT{n}})\) in \(\pi\) is in \(\Lambda\).
      Then we define \(\beta\) as
      \begin{multline*}
        \left(
          w,
          \AxiomC{\(\pi_0\)}
          \noLine
          \UnaryInfC{\(\necd^n \Sigma_0, \set{\nec^i \Sigma_i}_{1 \leq i < n}, \necd^n\Xi^{\nec}_0, \set{\nec^i \Xi^{\nec}_i}_{1 \leq i < n} \Rightarrow  \phi\)}
          \RightLabel{\(\modalalt{\wKT{n}}\)}
          \UnaryInfC{\(\set{\nec^{i+1} \Sigma_i}_{i < n}, \Gamma, \Xi \Rightarrow  \nec \Theta, \Delta, \nec \phi,\Lambda'\)}
          \DisplayProof
        \right)
        \overset{\beta}{\longmapsto}\\
        \AxiomC{\(\beta(w^{\nec},\pi_0)\)}
        \noLine
        \UnaryInfC{\((\kappa^{\nec})^*, \necd^n \Xi^{\nec}_0, \set{\nec^i \Xi^{\nec}_i}_{1 \leq i < n} \Rightarrow \phi \)}
        \RightLabel{\(\wk\)}
        \UnaryInfC{\(\necd^n (\kappa^{\nec})^*, \necd^n \Xi^{\nec}_0, \set{\necd^n\nec^i \Xi^{\nec}_i}_{1 \leq i < n} \Rightarrow \phi\)}
        \RightLabel{\(\modal{\wKT{n}}\)}
        \UnaryInfC{\(\nec (\kappa^{\nec})^*, \Xi \Rightarrow \nec \phi, \Lambda'\)}
        \RightLabel{\(\wk + \wedgeL\)}
        \doubleLine
        \UnaryInfC{\(\kappa^*_w, \Xi \Rightarrow \nec \phi, \Lambda' \)}
        \DisplayProof
      \end{multline*}
      where \(\mmod{\Sigma_i \union \Xi^{\nec}_i} \subseteq \set{[0]_n}\), \(\nec^{i+1} \Xi^{\nec}_{i} \subseteq \Xi\) and \(\Lambda = \nec \phi, \Lambda'\).
  \end{itemize}

  To show that \(\beta(w,\pi)\) is always a proof, we assign a measure \(\omega^2 |\Gamma_w \Rightarrow_{[i_w]_n} \Delta_w| + \omega\length(w) + \lhg(\pi)\) to every pair \((w,\pi)\), where \(|\Gamma_w \Rightarrow_{[i_w]_n} \Delta_w|\) is the sum of the complexity of the formulas in \(\Gamma_w \Rightarrow_{[i_w]_n} \Delta_w\) and \(\length(w)\) is the length of \(w\) as a sequence.
  We notice that in the corecursive calls of \(\beta\) this measure strictly decreases in all cases except when \(w\) has rule \((\modal[+]{\wKT{n}})\) and \(\pi\) ends in rule \((\modalalt{\wKT{n}})\).
  However, in these cases, at the produced preproof there is an application of \((\modal{\wKT{n}})\) in the path from the corecursive call to the root. 
  This implies that in any infinite branch of \(\beta(w, \pi)\), there will be infinitely many applications of \((\modal{\wKT{n}})\), as desired.
\end{proof}

Combing the two previous lemmas, we obtain the desired result.

\begin{theorem}
  \(\wGL{n}\) has uniform Lyndon interpolation.
\end{theorem}
\begin{proof}
  Let \(\phi\) be a formula and \(V_+,V_-\) be vocabularies.
  We know that \(\phi \Rightarrow\) has an interpolation template~\(T\).
  We have that \(\voc_+(\iota_T) \subseteq V_+\) and \(\voc_-(\iota_T) \subseteq V_-\) and since \(\n{\wGL{n}} \vdash \phi \Rightarrow \iota_T\) we get \(\wGL{n} \vdash \phi \to \iota_T\).
  Assume \(\wGL{n} \vdash \phi \to \psi\) for some \(\psi\) with \(\voc_b(\psi) \subseteq V_b\) for \(b \in \set{+,-}\).
  Then \(\nmod{\wGL{n}} \vdash \phi \Rightarrow \psi\) so \(\n{\wGL{n}} \vdash \iota_T \Rightarrow \psi\), from where we obtain the desired \(\n{\wGL{n}} \vdash \iota_T \to \psi\).
\end{proof}

\section*{Conclusion and future work}

In this paper we gave the first effective cut elimination proof for the sequent calculus of \(\g{\wGL{n}}\).
Additionally, we provided an alternative non-wellfounded sequent calculus that has an improved subformula property, which we used to provide a proof of uniform Lyndon interpolation for all Sacchetti's logics.
For this work we checked that, in \(\wGL{n}\), there is a class of modal equational systems where the depth of the variables at the solution can be adequately handled.

As future work, we leave open the option of exploring the notion of depth in interpolation.
It seems possible to consider a strengthening of interpolation, in which the depth of variables is part of the common vocabulary (instead of focusing on the polarity).
Our proofs suggest, although it should be checked in detail, that \(\wGL{n}\) has this alternative interpolation property (refining the equational system of interpolation templates to restrict properly the depth of variables, instead of allowing any depth).
Understanding the consequences of a logic having  this ``depth'' interpolation property is also left as future work.

\appendix
\section{Cut reductions}\label{sec:cut-reductions}
Here we display some cut reductions that were not shown in the proof of Theorem~\ref{cut-elimination-wGL}.

\textbf{Weakening part.}
If \(\chi\) belongs to the weakening part of the last rule instance of \(\pi\) or \(\tau\) we can delete directly, as weakening parts can be modified arbitrarily.
From now on we assume that \(\chi\) does not occur in the weakening part of the rule instances.

\textbf{Axiomatic}.
Assume \(\pi\) ends in \((\ax)\), the case for \(\tau\) is analogous.
Then \(\chi = p\) for some variable \(p\) and the desired reduction is
\[
  \AxiomC{\(\)}
  \RightLabel{\(\ax\)}
  \UnaryInfC{\(p, \Gamma' \Rightarrow \Delta, p\)}
  \DisplayProof
  \quad
  \AxiomC{\(\tau\)}
  \noLine
  \UnaryInfC{\(p, p, \Gamma' \Rightarrow \Delta\)}
  \DisplayProof
  \longmapsto
  \AxiomC{\(\ctr(\tau)\)}
  \noLine
  \UnaryInfC{\(p, \Gamma' \Rightarrow \Delta\)}
  \DisplayProof
\]
where \(\Gamma = p, \Gamma'\).

If \(\pi\) ends in \(\botL\) then \(\chi\) would belong to the weakening part, which is already covered.
Assume \(\tau\) ends in \(\botL\), so \(\chi = \bot\).
\[
  \AxiomC{\(\pi\)}
  \noLine
  \UnaryInfC{\(\Gamma \Rightarrow \Delta, \bot\)}
  \DisplayProof
  \quad
  \AxiomC{\(\tau\)}
  \noLine
  \UnaryInfC{\(\bot, \Gamma \Rightarrow \Delta\)}
  \DisplayProof
  \longmapsto
  \AxiomC{\(\mathrm{inv}_{\botR}(\pi)\)}
  \noLine
  \UnaryInfC{\(\Gamma \Rightarrow \Delta\)}
  \DisplayProof
\]

From now on we assume that neither \(\pi\) nor \(\tau\) end in \((\ax)\) or \((\botL)\).

\textbf{\(\botR\) case}.
Asssume \(\pi\) ends in an application of \((\botR)\), the case for \(\tau\) is analogous.
If \(\bot\) is the cut formula, the desired cut reduction is obtained taking the immediate subproof of \(\pi\).
If \(\bot\) is not the cut formula, the desired cut reduction is
\[
  \AxiomC{\(\pi_0\)}
  \noLine
  \UnaryInfC{\(\Gamma \Rightarrow \Delta', \chi\)}
  \RightLabel{\(\botR\)}
  \UnaryInfC{\(\Gamma \Rightarrow \bot, \Delta', \chi\)}
  \DisplayProof \quad
  \AxiomC{\(\tau\)}
  \noLine
  \UnaryInfC{\(\chi, \Gamma \Rightarrow \bot, \Delta'\)}
  \DisplayProof
  \longmapsto
  \AxiomC{\(\pi_0\)}
  \noLine
  \UnaryInfC{\(\Gamma \Rightarrow \Delta', \chi\)}
  \AxiomC{\(\mathrm{inv}_{\botR}(\tau)\)}
  \noLine
  \UnaryInfC{\(\chi, \Gamma \Rightarrow \Delta'\)}
  \RightLabel{\(\cut \text{(I.H.)}\)}
  \BinaryInfC{\(\Gamma \Rightarrow \Delta'\)}
  \DisplayProof
\]
where \(\Delta = \bot, \Delta'\).

From now on we assume that neither \(\pi\) nor \(\tau\) end in \((\botR)\).

\textbf{Principal cut reduction.}
Assume \(\chi\) is principal in \(\pi\) and \(\tau\).
Then \(\chi\) is an implication, as \(\nec\)-formulas cannot be principal at the left side of sequents.
\begin{multline*}
  \AxiomC{\(\pi_0\)}
  \noLine
  \UnaryInfC{\(\chi_0, \Gamma \Rightarrow \Delta, \chi_1\)}
  \RightLabel{\(\toR\)}
  \UnaryInfC{\(\Gamma \Rightarrow \Delta, \chi_0 \to \chi_1\)}
  \DisplayProof
  \quad
  \AxiomC{\(\tau_0\)}
  \noLine
  \UnaryInfC{\(\Gamma \Rightarrow \Delta, \chi_0\)}
  \AxiomC{\(\tau_1\)}
  \noLine
  \UnaryInfC{\(\chi_1, \Gamma \Rightarrow \Delta\)}
  \RightLabel{\(\toL\)}
  \BinaryInfC{\(\chi_0 \to \chi_1, \Gamma \Rightarrow \Delta\)}
  \DisplayProof\\
  \\
  \longmapsto
  \AxiomC{\(\wk(\tau_0)\)}
  \noLine
  \UnaryInfC{\(\Gamma \Rightarrow \Delta, \chi_1, \chi_0\)}
  \AxiomC{\(\pi_0\)}
  \noLine
  \UnaryInfC{\(\chi_0, \Gamma \Rightarrow \Delta, \chi_1\)}
  \RightLabel{\(\cut \text{(I.H.)}\)}
  \BinaryInfC{\(\Gamma \Rightarrow \Delta, \chi_1\)}
  \AxiomC{\(\tau_1\)}
  \noLine
  \UnaryInfC{\(\chi_1, \Gamma \Rightarrow \Delta\)}
  \RightLabel{\(\cut \text{(I.H.)}\)}
  \BinaryInfC{\(\Gamma \Rightarrow \Delta\)}
  \DisplayProof
\end{multline*}

\textbf{Commutative cut reduction.}
Finally, assume that the cut formula is not principal in either \(\pi\) or \(\tau\).
Assume that the cut formula is not principal in \(\pi\) and the last rule of \(\pi\) is \((\toR)\), the cases where the rule is instead \((\toL)\), or where any of theses occurs at \(\tau\) with the cut formula non-principal are analogous.
The desired cut reduction is

\[
  \AxiomC{\(\pi_0\)}
  \noLine
  \UnaryInfC{\(\phi, \Gamma \Rightarrow \psi, \Delta', \chi\)}
  \RightLabel{\(\toR\)}
  \UnaryInfC{\(\Gamma \Rightarrow \phi \to \psi, \Delta', \chi\)}
  \DisplayProof
  \quad
  \AxiomC{\(\tau\)}
  \noLine
  \UnaryInfC{\(\chi, \Gamma \Rightarrow \phi \to \psi, \Delta'\)}
  \DisplayProof
  \longmapsto
  \AxiomC{\(\pi_0\)}
  \noLine
  \UnaryInfC{\(\phi, \Gamma \Rightarrow \psi, \Delta', \chi\)}
  \AxiomC{\(\mathrm{inv}_{\toR}(\tau)\)}
  \noLine
  \UnaryInfC{\(\chi, \phi, \Gamma \Rightarrow \psi, \Delta'\)}
  \RightLabel{\(\cut\text{ (I.H.)}\)}
  \BinaryInfC{\(\phi, \Gamma \Rightarrow \psi, \Delta'\)}
  \RightLabel{\(\toR\)}
  \UnaryInfC{\( \Gamma \Rightarrow \phi \to \psi, \Delta'\)}
  \DisplayProof
\]
where \(\Delta = \phi \to \psi, \Delta'\)


Finally, we have that the cut formula is not principal in either \(\pi\) or \(\tau\) with last rule \((\modal{\wKT{n}})\).
We notice that this cannot occur in \(\pi\), as then the cut formula would belong to the weakening part, so it must occur in~\(\tau\).
To make the cut formula not belong to the weakening part, it must be the case that \(\chi = \nec \chi_0\).
If \(\chi\) were not principal in \(\pi\), then we would be in one of the previous cases. Hence, we can assume that \(\chi\) is prinicipal in \(\pi\), and thus \(\pi\) must end in \((\modal{\wKT{n}})\).
This case is covered in the proof of Theorem~\ref{cut-elimination-wGL}.

\bibliography{bib}
\bibliographystyle{abbrvnat}
\end{document}